\documentclass{amsart}

\usepackage{amssymb}
\usepackage{rotating}

\theoremstyle{plain}

\numberwithin{equation}{section}

\newtheorem{theo}{Theorem}

\newtheorem{lemma}{Lemma}
\newtheorem{cora}[lemma]{Corollary}

\newtheorem{prop}[lemma]{Proposition}

\theoremstyle{definition}
\newtheorem{definition}[lemma]{Definition}

\newtheorem{remark}{Remark}   
\newtheorem{example}{Example}

\newcommand{\stand}{\operatorname{\mathfrak{s}}}

\theoremstyle{definition}

\newcommand{\comment}[1]{}

\newcommand{\somecycle}{\cdot\beta_2(m_1\  \cdots)(m_2\ m_3\  \cdots)(m_4\ \  \cdots\ n){]}&=&}

\usepackage{fancyhdr,ifthen}
\begin{document}	
\title{Products of conjugacy classes of the alternating group}

\author{Edith Adan-Bante, John Harris, and Helena Verrill}
\thanks{H. Verrill is partially supported by NSF grant DMS-0501318.}

\address{Department of Mathematical Science, Northern Illinois University, DeKalb, Ilinois, 60115-2895}
\email{EdithAdan@illinoisalumni.org}

\address{University of Southern Mississippi Gulf Coast, Long Beach, MS 39560} 
\email{john.m.harris@usm.edu}

\address{Department of Mathematics, Louisiana State University, Baton Rouge, LA 70803-4918}
\email{verrill@lsu.edu}



\date{June 15, 2009}
\begin{abstract} Let $A_n$ be the alternating group on $n$ letters. For $n>5$, we  describe the elements $\alpha,\beta\in A_n$ when 
$\alpha^{A_n}\beta^{A_n}$ is the union of at most four distinct conjugacy 
classes.
\end{abstract}
\maketitle

\begin{section}{Introduction}

Let $G$ be a finite group.
Then the product of two conjugacy classes $a^G$, $b^G$ in $G$ is
the union of 
$m$ distinct conjugacy classes of $G$, for some integer $m>0$. We set
 $\eta(a^Gb^G)=m$.
 In this note, we continue our exploration of the minimum possible value of
 $\eta(\alpha^G \beta^G)$, begun in \cite{symmetric}, 
investigating the case where
$G$ is the symmetric or alternating group on $n$ letters.

Throughout this paper we denote by $\min(n)$ the smallest integer in 
the set $\{\eta(\alpha^{A_n}\beta^{A_n})\mid \alpha, \beta \in A_n\setminus\{e\}\}$.
It is known %\cite{?} 
that for $n\geq 5$, 
the product of two conjugacy classes of $A_n$ 
is never a conjugacy class, so for $\alpha, \beta$ nontrivial,
$\eta(\alpha^{A_n}\beta^{A_n})\ge2$.  On the other hand,
for any $n>5$, we can  check that
 $\eta((1\ 2 \ 3)^{A_n} (1\ 2\ 3)^{A_n})=5$, so
$2\le \min(n)\le 5$.
The main result of this paper is the following theorem.

\begin{theo}
\label{thm:maintheorem}
 Fix an integer $n\geq 6$.
 Let $A_n$ be the alternating group on $n$ letters and $\alpha,\beta\in A_n$ be nontrivial elements. Assume that
 $\eta(\alpha^{A_n} \beta^{A_n})<5$.  Then either $\alpha$ or
 $\beta$ is a product of transpositions. Assume that $\alpha$ is a product of transpositions. Then
 one of the following holds:
 
 i)  $n$ is a multiple of 4,  $\alpha$ is the product of 
 $\frac{n}{2}$ disjoint 2-cycles and $\beta$ is 
 a 3-cycle, and $\eta(\alpha^{A_n} \beta^{A_n})=2$.
 
 ii) $n-1$ is a multiple of 4,  $\alpha$ is the product of 
 $\frac{n-1}{2}$ disjoint 2-cycles and $\beta$ is 
 a 3-cycle, and $\eta(\alpha^{A_n} \beta^{A_n})=4$.
\end{theo}

Since for any group $G$ and any $a, b \in G$, we have that 
$a^Gb^G=b^Ga^G$, the previous result describes all the possible $\alpha, \beta\in A_n\setminus \{e\}$ such that $\eta(\alpha^{A_n} \beta^{A_n})<5$.

\begin{cora}
For $n\ge 6$, 
$$
\min(n)=
\begin{cases}
2 \text{ if } n\equiv 0 \mod 4\\
4 \text{ if } n\equiv 1 \mod 4\\
5 \text{ otherwise}
\end{cases}
$$
\end{cora}

\begin{remark}
Values of $\min(n)$ for $n\le 12$ can be computed using a computer program
such as GAP \cite{GAP4}, and are given in the tables in the Appendix.
%Section~\ref{subsec:tables}.
\end{remark}
\begin{remark}
For most pairs $\alpha, \beta\in A_n$,
$\eta(\alpha^{A_n} \beta^{A_n})$ is much larger than $\min(n)$
as shown in the tables in the Appendix.
\end{remark}

{\bf Acknowledgment.} The first author would like to thank FEMA for providing her
 with temporary housing in the  aftermath of hurricane Katrina. 

\end{section}
\begin{section}{Notation}
 
We use the standard cycle notation for permutations, and we take
our maps to be written on the right, so we multiply cycles from
left to right.  Thus for example,
$$(1\ 2\ 4)(1\ 2\ 3)(4\ 5\ 6)=(1\ 3)(2\ 5\ 6\ 4).$$
For integers $n >m$, 
any element of $A_m$ can also be considered as an element of $A_n$,
so we use this fact without comment.

The following is a very well known result.
\begin{lemma}\label{lemma1}
Let $\alpha \in S_n$. Then

a) $\alpha$ can be written as a product of disjoint cycles.

b) $\alpha^{S_n}$ is the set of all permutations of $S_n$ with the same
cycle structure as $\alpha$.
\end{lemma}

%Now we define our 
%standard ordering of the cycle type of $\alpha$.
%, and a standard element of any conjugacy class of $S_n$.
\begin{definition}
\label{def:standard}
The cycle type of a element of $S_n$ is an \underline{unordered}
multiset
of integers, forming a partition of $n$.
We use square bracket notation, for example,
For $\alpha\in S_n$, the
cycle type of $\alpha$ can be written as
%be the ordered sequence
\begin{equation}
\label{eqn:standardcycletype}
[\alpha]=[
\underbrace{1,\cdots,1}_{n_0}, 
\underbrace{a_1,\cdots,a_1}_{n_1}, 
\underbrace{a_2,\cdots,a_2}_{n_2}, \ldots, 
\underbrace{a_r,\cdots,a_r}_{n_r}],
\end{equation}
where $\alpha$ consistes of disjoint cycles of lengths $a_1,\dots,a_r$,
$$
\begin{array}{l}
r\ge 0 ,\\
n_i\ge 1\text{  for  }1\le i\le r, \text{ and } n_0\ge 0\\
1<a_1<a_2<\cdots< a_{r-1}<a_r,\\
n_0+\sum_{i=1}^{r}n_ia_i=n.\\
\end{array}
$$
i.e., here the cycle lengths are written in increasing order of size,
and form a partition of $n$.
\end{definition}

\begin{example}
For the following element of $A_{12}$ we have
\begin{eqnarray*}
{[}(8\ 9)(5\ 6\ 7\ 1)(11)(12){]}&=&{[}1,1,1,1,1,1,2,4{]}
\end{eqnarray*}
\end{example}

Lemma \ref{lemma1}
says that
 two elements $\gamma$, $\delta$ in $S_n$ are conjugate
 if and only if they are of the same type, that is if and only if
 $[\gamma]=[\delta]$.
 
\end{section}
\section{Proofs}

\subsection{Reducing from $A_n$ to $S_n$}

Although our main result is about conjugacy classes in $A_n$, by
means of the following results, we will reduce most of the work to
determining pairs $\alpha,\beta\in S_n$
for which $\eta(\alpha^{S_n}\beta^{S_n})<5$.
Note that for $\alpha\in A_n$, if $\alpha^{A_n}\not=\alpha^{S_n}$,
then $\alpha^{S_n}$ is the union of two $A_n$ conjugacy classes.

We will make use of the following two well known lemmas;
the proofs are included for completeness.

\begin{lemma}\label{oneodd} 
% note, oneodd in refs needs to
% be checked as to whether it refers to corol. ??? similarly for next corol
Let $\alpha\in A_n$ and $\beta\in S_n\setminus A_n$. Assume that
$\alpha^\beta=\beta^{-1}\alpha\beta=\alpha$. Then all the permutations of type 
$[\alpha]$ are $A_n$-conjugate to $\alpha$,
i.e., 
 $\alpha^{A_n}=\alpha^{S_n}$.
Conversely, suppose that
 $\alpha^{A_n}=\alpha^{S_n}$, then
$\alpha=\alpha^{\beta}$ for some $\beta\in S_n\setminus A_n$.
\end{lemma}
\begin{proof} 
First assume that $\alpha=\alpha^\beta$ for some odd $\beta$.
Let $\gamma$ be any permutation with $[\gamma]=[\alpha]$. By
Lemma~\ref{lemma1} there exists $\delta\in S_n$ such that
$\delta^{-1}\alpha \delta=\gamma$. If $\delta\in A_n$, then 
$\gamma$ and $\alpha$ are $A_n$-conjugates. Otherwise $\beta\delta\in A_n$.
Observe that $(\beta\delta)^{-1}\alpha (\beta\delta)=\gamma$ and thus
$\alpha$ and $\gamma$ are $A_n$-conjugates.

Suppose that
 $\alpha^{A_n}=\alpha^{S_n}$.
So for some element $\gamma\in S_n\setminus A_n$, and
some $\delta\in A_n$, we have
$\alpha^\gamma=\alpha^\delta$.  So
$\alpha^{\gamma\delta^{-1}}
=\alpha$, so we can take 
$\beta={\gamma\delta^{-1}}$.
\end{proof}

\begin{cora}
\label{coratooneodd}
Let $\alpha\in A_n$.
If for some integer
$m>0$ and some permutation $\alpha_1\in A_{n-2m}$ 
the cycle type of $\alpha$ differs from that of $\alpha_1$ by
insertion of either  $\{2m\}$ or $\{m,m\}$,
then $\alpha^{A_n}=\alpha^{S_n}$.
\end{cora}
\begin{proof}
If  $\alpha=\alpha_1\gamma$, where $\gamma$ is a cycle of order $2m$
 and $\gamma$ and $\alpha_1$ are disjoint permutations, then 
 $\alpha^{\gamma}=\alpha$. 
Since a cycle of even order is an odd permutation, the  result follows
from Lemma~\ref{oneodd}, since
$\alpha^\gamma=\alpha$.
 
 Assume now that $\alpha=\gamma_1\gamma_2\alpha_1$, where $\gamma_1$ and $\gamma_2$ are
 cycles of  order $m$ and $\alpha_1$, $\gamma_1$ and $\gamma_2$ are disjoint 
 permutations. If $m$ is even, 
the result then follows by the previous paragraph.
 We may assume then that $m$ is odd. Without loss of generality, we may 
 assume that $\gamma_1\gamma_2=(1\ 2 \ \cdots \ m)(m+1 \ m+2 \cdots \ 2m)$.
 Observe that the permutation $\delta=(1\ m+1)(2\ m+2)\cdots (i\ m+i)\cdots (m\  2m)$
 is an odd permutation
 and $(\gamma_1\gamma_2)^{\delta}=\gamma_2\gamma_1$. Thus 
 $\alpha^{\delta}=\alpha$ 
  and the result follows. 
\end{proof}

\begin{lemma}\label{splittingintotwo}
Let $\alpha\in A_n$. Set $[\alpha]=[a_1,\ldots, a_r]$.
Then $\alpha^{S_n}$ is the union of  two
distinct  conjugacy classes of
$A_n$ 
(equivalently, $\alpha^{S_n}\not=\alpha^{A_n}$)
if and only if $a_i$ is an odd integer for all $i$ and 
$a_i\neq a_j$ for any $i\neq j$. 
\end{lemma}
\begin{proof}
By Lemma \ref{oneodd}, notice that if some  $a_i$ is even for some $i$,
 then all the elements
of type $[a_1,\ldots, a_r]$ are $A_n$-conjugate,
since we can just conjugate by the corresponding cycle of length $a_i$.
 If there exist $a_i$ and $a_j$ such 
that $a_i=a_j$, then $\alpha$ will be fixed by an odd permutation,
as in the second case in Corollary~\ref{coratooneodd}.

Conversely, assume that $a_i$ is an odd integer for all $i$ and 
$a_i\neq a_j$ for any $i\neq j$. By renaming if necessary, we may 
assume that $\alpha=(1\ 2\ \cdots\ a_1)\alpha_1$,
with $\alpha_1$ fixing $1, 2,\dots,a_1$.
Let
$\beta=\alpha^{(1\ 2)}=(2\ 1\ 3\cdots\ a_1)\alpha_1$.
 We can check that since 
$a_i\neq a_j$ and $a_i$ is odd for all
$i$, if $\alpha^{\gamma}=\beta$, then $\gamma=(1\ 2)$. Thus
$\alpha$ and $\beta$ are not $A_n$ conjugates but clearly $[\alpha]=[\beta]$.
\end{proof}

\begin{cora}
\label{cora:splittingintotwo}
Let $\alpha\in A_n$.
Suppose that $\alpha^{S_n}\not=\alpha^{A_n}$.
Then $\alpha$ fixes at most one 
element. If $\alpha=(1\ 2\ 3 \ \cdots\ a_1)\alpha_1$,
where $\alpha_1$ fixes $1,2,\dots,a_1$, then
$\alpha^{S_n}=\alpha^{A_n} \cup (\alpha^{(1\ 2)})^{A_n}
=\alpha^{A_n} \cup 
((2\ 1\ 3 \ \cdots\ a_1)\alpha_1)^{A_n}$. 
\end{cora}

\begin{prop}
\label{prop:reducefromAntoSn}
Let $\alpha,\beta\in A_n$,
and suppose that
either 
$\alpha^{A_n}=\alpha^{S_n}$ or $\beta^{A_n}=\beta^{S_n}$.
  Then 
$\eta(\alpha^{A_n}\beta^{A_n})\geq 
\eta(\alpha^{S_n}\beta^{S_n})$.
\end{prop}
\begin{proof}
Suppose that $\alpha^{A_n}=\alpha^{S_n}$.
Then by Lemma~\ref{oneodd} there is some odd permutation 
$\gamma$ with $\alpha^\gamma=\alpha$.
For any $a,b\in S_n$, one of the pairs
$(a,b)$, 
$(\gamma a,b)$, $(\gamma a\gamma^{-1},b\gamma^{-1} )$, 
$(a\gamma^{-1},b\gamma^{-1} )$, is in $A_n\times A_n$.
Since $\alpha^\gamma=\alpha$, we have
$[\alpha^a\beta^b]=[\alpha^{\gamma a}\beta^b]=
[(\alpha^{\gamma a}\beta^b)^{\gamma^{-1}}]=
[\alpha^{\gamma a\gamma^{-1}}\beta^{b\gamma^{-1}}]$
and 
$
[\alpha^a\beta^b]=
[(\alpha^a\beta^b)^{\gamma^{-1}}]
[\alpha^{a\gamma^{-1}}\beta^{b\gamma^{-1}}]
$, so
 $\alpha^{A_n}\beta^{A_n}$ contains elements of all possible cycle types
of elements of $\alpha^{S_n}\beta^{S_n}$, so the result follows.
Similarly if $\beta^{A_n}=\beta^{S_n}$.
\end{proof}

\begin{remark}
Experimentally, it seems that
$\eta(\alpha^{A_n}\beta^{A_n})\geq 
\eta(\alpha^{S_n}\beta^{S_n})$
even when 
$\alpha^{A_n}\not=\alpha^{S_n}$ and  $\beta^{A_n}\not=\beta^{S_n}$,
but we have not been able to prove this.
\end{remark}

In Proposition~\ref{prop:reducefromAntoSn}, we have to assume either
$\alpha^{A_n}=\alpha^{S_n}$ or $\beta^{A_n}=\beta^{S_n}$.
We now deal with the case where 
$\alpha^{A_n}\not=\alpha^{S_n}$ and  $\beta^{A_n}\not=\beta^{S_n}$.

\begin{lemma}\label{order} Let $G$ be a finite group, $a$ and $b$ in $G$.
Then $a^G b^G =b^G a^G$.
\end{lemma}

\begin{lemma}
\label{wasclaim}
Let $\alpha,\beta\in A_n$ with $n\ge 7$.
Assume that
\begin{eqnarray*}
\alpha&=&(1 \ 2 \ 3\ 4 \ 5\ 6 \cdots \ a_1)\gamma, \\
\beta&=&(1 \ 2 \ 3\ 4 \ 5\ 6 \cdots \ b_1)\delta,
\end{eqnarray*}
with $a_1\ge 7$ and $\gamma$ fixing $1,2,\dots,a_1$, and
with $b_1\ge 7$ and $\delta$ fixing $1,2,\dots,b_1$.
Then
$\eta(\alpha^{A_n}\beta^{A_n})\ge 5$
and 
$\eta(\alpha^{A_n}\beta^{(2\ 4)(3\ 6\ 5)A_n})\ge 5$
\end{lemma}
\begin{proof}
Let
\begin{eqnarray*}
\alpha_2&=&(1 \ 3 \ 2\ 7 \ 5\ 6\ 4\ \cdots \ a_1)\gamma,\\
\alpha_3&=&(1 \ 3\ 2 \ 5 \ 4\ 6 \ 7 \cdots \ a_1)\gamma, \\
\alpha_4&=&(1 \ 5 \ 4\ 3 \ 2\ 6 \ 7 \cdots \ a_1)\gamma,\\
\alpha_5&=&(1 \  6 \ 5\ 4\ 3\ 2\ 7\cdots\  a_1)\gamma.
\end{eqnarray*}
Then $\alpha$, $\alpha_1$, $\alpha_2$, $\alpha_3$ and $\alpha_4$ are $A_n$-conjugates.
We will now show that
 $\alpha\beta$, $\alpha_2\beta$,$\alpha_3\beta$,$\alpha_4\beta$,
$\alpha_5\beta$
fix different number of points, and thus are not conjugate to each other.

Let $X$ be the set of fixed points of $\alpha\beta$.
 Observe that $\{1,2,3,4,5\}\cap X=\emptyset$. 
 We can check that the set of fixed points of 
 $\alpha_2\beta$, $\alpha_3\beta$, $\alpha_4\beta$ and 
$\alpha_5\beta$
 is $X\cup\{3\}$, $X\cup\{3,4\}$ 
 $X\cup\{2,3,4\}$, and $X\cup \{2,3,4,5\}$
  respectively. 

Our proof is similar when $\beta$ is replaced by 
$$\beta':=\beta^{(2\ 4)(3\ 6\ 5)}=
(1\ 4\ 6\ 2\ 3\ 5\ 7\ \dots\ b_1)\delta.$$
Define a further conjugate of $\alpha$ in $A_n$ by
$$\alpha_6=\alpha^{(2\ 7\ 4\ 3\ 5)}=(1\ 7\ 5\ 3\ 2\ 6\ 4\ \cdots\ a_1)\gamma.$$
Now if $Y$ is the set of fixed points of $\alpha\beta'$,
then the sets of fixed points of 
$\alpha_2\beta'$,$\alpha_3\beta'$,$\alpha_4\beta'$,
$\alpha_5\beta',\alpha_6\beta'$
are
$Y\cup \{3,6,7\},
Y\cup \{3\},
Y\cup \{2,3\},
Y\cup \{3\},
Y\cup \{2,3,6,7\}.
$
respectively.
The numbers of fixed points distinguishes the conjugacy classes of
$\alpha\beta'$
$\alpha_2\beta'$,$\alpha_3\beta'$,$\alpha_4\beta'$ and
$\alpha_6\beta'$ 
from each other.
\end{proof}

\begin{lemma}\label{bothsplitting}
Let $\alpha,\beta\in A_n$ with $n\geq 6$. Assume that  $\alpha^{A_n}\neq \alpha^{S_n}$ and
 $\beta^{A_n}\neq \beta^{S_n}$. Then 
$\eta(\alpha^{A_n}\beta^{A_n})\geq 5$.
\end{lemma}
\begin{proof}
By Lemma~\ref{splittingintotwo}
and
Corollary~\ref{cora:splittingintotwo}, we have that both 
$\alpha$ and $\beta$ fix at most one element 
and they are  products of disjoint cycles of different odd order.

By Lemma~\ref{splittingintotwo} and computations with GAP, we may
assume that both $\alpha$ and $\beta$ have cycles of length at 
least $7$.  This follows, since by Lemma~\ref{splittingintotwo},
we know that $\alpha$ and $\beta$ must have cycles of distinct odd
lengths.
If there is no cycle of length at least $7$, then $n=1,4,8$ or $9$.
However, we assume $n\ge 6$, so we must have $n=8$ or $9$, and
the cycle types of $\alpha$ and $\beta$ must be either
$[1,5]$, $[3,5]$, or $[1,3,5]$.
From Tables~\ref{table:eta_for_A6},
\ref{table:eta_for_A8} and
\ref{table:eta_for_A9}
 we see that in these cases, we have
$\eta(\alpha^{A_n}\beta^{A_n})=7$, $13$ and $17$ respectively.
and so we already have 
$\eta(\alpha^{A_n}\beta^{A_n})\ge 5$ in these cases.

By Lemma~\ref{cora:splittingintotwo}
and Lemma~\ref{order},  
 and renaming if 
necessarily, we may assume that 
\begin{equation}
\alpha=(1 \ 2 \ 3\ 4 \ 5\ 6 \cdots \ a_1)\gamma, 
\end{equation}
\noindent and
\begin{equation}
\mbox{ either } \beta=(1 \ 2 \ 3\ 4 \ 5\ 6 \cdots \ b_1)\delta \mbox{ or }
\beta=(1\ 2\ 3\ 4 \ 5\ 6 \cdots \ b_1)^\mu\delta,
\end{equation}
\noindent for some permutations $\gamma$ and $\delta$, where
$a_1\ge 7$, $b_1\ge 7$ and $\gamma$ fixes $1,\dots,a_1$
and $\delta$ fixes $1,\dots,b_1$, and $\mu$ is any odd permutation.
These two cases are dealt with by Lemma~\ref{wasclaim}.
Note that by conjugation in $A_n$
we can take any fixed choice of odd $\mu$, and a convenient choice is
$\mu=(2\ 4)(3\ 6\ 5)$.
\end{proof}

\subsection{An inductive argument}

We would like to apply an inductive argument in this section.
We'd like to be able to pass from $\alpha^{S_n}\beta^{S_n}$
to something akin to $\alpha^{S_{n-1}}\beta^{S_{n-1}}$.
Instead of using $\eta$, we will use $\eta'$ defined below.
To define this we need to choose a particular element,
$\stand(\alpha)$ of any
conjugacy class $\alpha^{S_n}$ of $S_n$.
Note that $\eta'$ and $\stand$ are only used in this subsection.

\begin{definition}
Suppose that $\alpha\in S_n$ has cycle structure as in 
(\ref{eqn:standardcycletype}).
We define the element $\stand(\alpha)$ of $\alpha^{S_n}$ by
$$\stand(\alpha)=
(1+n_0\ \dots\ a_1+n_0)(a_1+n_0+1\ \dots\ 2a_1+n_0)\cdots(n-a_r+1\dots\ n).$$
I.e., the cycles are written in increasing order of length, and the
elements appearing are written in increasing consecutive order.
Elements $1,\dots,n_0$ are fixed.  The element $n$ is in a cycle of
maximum length, so unless $\alpha=e$, $\stand(\alpha)$ never
fixes $n$.
\end{definition}

\begin{example}
For the following element of $A_{12}$ we have
\begin{eqnarray*}
\stand(8\ 9)(5\ 6\ 7\ 1)(11)(12)&=&
(1)(2)(3)(4)(5)(6)(7\ 8)(9\ 10\ 11\ 12).
\end{eqnarray*}
\end{example}

\begin{definition}
\label{defmodifiedeta}
Let $S_n$ be the group of permutations on the set $\{1,\dots,n\}$.
Define $\eta'(a,b)$ for $a,b\in S_n$ by
$$\eta'(a,b)=
\#\{[(\stand(a)^{-1})^\sigma \stand(b)] 
\;:\; \sigma \in S_n,\; \sigma \text{ fixes } n\}.$$
This definition is choosen because in the inductive step, 
(details in Proposition~\ref{prop:inductionstep}),
as well as having $(n)\sigma=n$,
we will also take
$(n-1)\sigma=(n-1)$. In this case,
provided both $\alpha$ and $\beta$ are not the identity, we have
$(n-1)\stand(\alpha)=n$, $(n-1)\stand(\beta)=n$,
and $(n)\stand(\alpha)^{-1}=n-1$,
so 
\begin{eqnarray}
(n)(\stand(\alpha)^{-1})^\sigma\stand(\beta)&=&
(n)\sigma^{-1}\stand(\alpha)^{-1}\sigma\beta
=(n)\stand(\alpha)^{-1}\sigma\beta\\
&=&(n-1)\sigma^{-1}\stand(\beta)
\nonumber
=(n-1)\stand(\beta)=n.\end{eqnarray}
\end{definition}

\begin{lemma}
\label{lem:comparisonofetas}
With $a,b\in S_n$ %as in Definition~\ref{defmodifiedeta},
$$\eta(a^{S_n}b^{S_n})\ge \eta'(a,b).$$
\end{lemma}
\begin{proof}
For any $\beta\in S_n$, 
$\beta$ is conjugate to $\beta^{-1}$ and to $s(\beta)$.
So since $\alpha^\sigma\beta^\tau$ is conjugate to 
$\alpha^{\sigma\tau^{-1}}\beta$,
and since $\alpha^{S_n}$ depends on $\alpha$ only up to conjugacy,
we have 
$$\eta(a^{S_n}b^{S_n})=
\#\{[(\stand(a)^{-1})^\sigma \stand(b)] \;:\; \sigma \in S_n\}.$$
The set involved contains the set used to define $\eta'$, and so
the inequality follows.
\end{proof}

Thus in order to find a lower bound on $\eta(a^{S_n}b^{S_n})$, 
it suffices to find
a lower bound on $\eta'(a,b)$.

\begin{prop}
\label{prop:inductionstep}
Fix an integer $n>1$.
Let $\alpha,\beta\in S_n\setminus{e}$ with cycle structure
\begin{eqnarray*}
[\alpha]&=&[
\underbrace{1,\cdots,1}_{n_0}, 
\underbrace{a_1,\cdots,a_1}_{n_1}, 
\underbrace{a_2,\cdots,a_2}_{n_2}, \ldots, 
\underbrace{a_r,\cdots,a_r}_{n_r}]\\
{[}
\beta ]&=&[
\underbrace{1,\cdots,1}_{m_0}, 
\underbrace{b_1,\cdots,b_1}_{m_1}, 
\underbrace{b_2,\cdots,b_2}_{m_2}, \ldots, 
\underbrace{b_s,\cdots,b_s}_{m_s}]\\
\end{eqnarray*}

where
\begin{equation}
\begin{array}{l}
r,s\ge 1 , n_i,m_j\ge 1 , \\
n_i\ge 1\text{  for  }1\le i\le r , m_i\ge 1\text{  for  }1\le i\le s ,\\
1<a_1<a_2<\cdots< a_{r-1}<a_r , 1<b_1<b_2<\cdots<b_{s-1}<b_s ,\\
\sum_{i=1}^{r}n_ia_i=\sum_{i=1}^{s}m_ib_i=n ,\\
\alpha(2)=1 , \; \beta(1)=2 ,\\
\alpha^{a_1}(2)=\beta^{a_1}(2)=2.
\end{array}
\end{equation}

Then there exist
$\alpha',\beta'\in S_{n-1}$ with
cycle structure
\begin{eqnarray*}
[\alpha']&=&[
\underbrace{1,\cdots,1}_{n_0}, 
\underbrace{a_1,\cdots,a_1}_{n_1}, 
\underbrace{a_2,\cdots,a_2}_{n_2}, \ldots, 
\underbrace{\cdots,a_r-1}_{+1},
\underbrace{a_r,\cdots,a_r}_{n_r-1}]\\
{[}
\beta' ]&=&[
\underbrace{1,\cdots,1}_{m_0}, 
\underbrace{b_1,\cdots,b_1}_{m_1}, 
\underbrace{b_2,\cdots,b_2}_{m_2}, \ldots, 
\underbrace{\cdots,b_s-1}_{+1},
\underbrace{b_s,\cdots,b_s}_{m_s-1}]
\end{eqnarray*}
(i.e., there is one fewer cycle of maximal length, and one more cycle of
length one less than this)
such that
\begin{equation}
\label{induction1}
\eta'(\alpha,\beta)\ge \eta'(\alpha',\beta').
\end{equation}
\end{prop}
\begin{proof}
Note that since $\alpha,\beta\not=e$, 
$a_r,b_s>2$, so $n$ is not a fixed point.  
Since $\eta'(a,b)=\eta(\stand(a),\stand(b))$,
we may assume for simplicity that $\alpha=s(\alpha)$ and $\beta=s(\beta)$.

Suppose that
\begin{eqnarray*}
\alpha&=&\stand(\alpha)=\alpha_2(u\ \cdots\ n-2\ n-1\ n),\\
\beta&=&\stand(\beta)  =\beta_2(v\ \cdots\ n-2\ n-1\ n),
\end{eqnarray*}
where $u=n-a_r, v=n-b_s$, and $\alpha_2$ fixes
$u,\dots,n$ and $\beta_2$ fixes $v,\dots,n$.

Now take $\sigma\in S_n$ with $(n)\sigma=n$ and $(n-1)\sigma=n-1$.
Then
\begin{eqnarray*}
(\alpha^{-1})^\sigma\beta
&=&
(\alpha_2^{-1})^\sigma\Big(n\ n-1\ \sigma(n-2)\ \cdots\ 
\sigma(u)\Big)
\beta_2\Big(v\  \cdots\ \ n-1\ n\Big)
\\
&=&
(n)(\alpha_2^{-1})^\sigma
\Big(n-1\ \sigma(n-2)\ \cdots\ 
\sigma(u)\Big)
\beta_2\Big(v\  \cdots\ \ n-1\Big).
\end{eqnarray*}
So setting
\begin{eqnarray*}
\alpha'&=&\alpha_2\Big(u\ \cdots\ n-2\ n-1\Big)\\
\beta'&=&\beta_2\Big(v\ \cdots\ n-2\ n-1\Big),
\end{eqnarray*}
we have
$$(\alpha^{-1})^\sigma\beta=((\alpha')^{-1})^{\sigma}(\beta')^{-1}.$$
Note that since $\stand(\alpha)=\alpha$ and
$\stand(\beta)=\beta$, we also have
$\stand(\alpha')=\alpha'$ and
$\stand(\beta')=\beta'$ (as elements of $S_{n-1}$).
So
\begin{eqnarray*}
\eta'(\alpha',\beta')
&=&\#\{[((\alpha')^{-1})^\sigma\beta' | \sigma\in S_{n-1}, \sigma
\text{ fixes }n-1\}\\
&=&
\#\{[((\alpha)^{-1})^\sigma\beta | \sigma\in S_{n}, \sigma
\text{ fixes }n-1 \text{ and }n\}\\
&\le&
\#\{[((\alpha)^{-1})^\sigma\beta | \sigma\in S_{n}, \sigma
\text{ fixes } n\}\\
&=&
\eta'(\alpha,\beta).
\end{eqnarray*}
Thus the inequality (\ref{induction1}) holds.
\end{proof}
%Remark: probably in most cases (\ref{induction1}) 
%will be a strict inequality.  If this could be proved in certain cases,
%then we would know more about bounds on how $\eta$ grows.

\begin{lemma}
\label{lem:fourcyclemodifiedeta}
For $\alpha=(1\ 2\ 3\ 4)$ and $\beta\in S_n$ with $n\ge 8$,
and $\beta(1)\not=1$,
 we have
$$\eta'(\alpha,\beta)\ge 5$$
\end{lemma}
\begin{proof}
Suppose that $\beta$ consists of at least $3$ cycles,
at least $2$ of which are nontrivial.
Suppose with the cycle lengths written in increasing order,
with longest cycles of lengths $b_1,b_2,b_3$, we have
$$[\beta]=[S,b_1,b_2,b_3],$$ 
where $S$ is some sequence of integers,
and $b_2,b_3\ge 2$.
We can assume $\beta=\stand(\beta)$, 
and the cycles of the lengths $b_1,b_2,b_3$ start
$(m_1\ \cdots)$,
$(m_2\ m_3\ \cdots)$,
$(m_4\ \ \cdots\ n)$, so
$$\beta=\stand(\beta)=
\beta_2
(m_1\ \cdots)
(m_2\ m_3\ \cdots)
(m_4\ \ \cdots\ n),$$
where $\beta_2$ fixes $1,\dots,m_1-1$,
and $m_1=n-b_1-b_2-b_3$,
$m_2=n-b_2-b_3$, $m_3=m_2+1$, and $m_4=n-b_3$. 

Then we have at least the following possible cycle
types of elements of the form $(\stand(\alpha)^{-1})^\sigma\stand(\beta)$, with $\sigma(n)=n$:
\begin{eqnarray*}
{[}( n\ m_1 \ m_2 \ m_3 ) \somecycle   b_1 + b_3 +1,   b_2-1,S\\
{[}( n\ m_1 \ m_2 \ m_4 ) \somecycle   b_1+b_2+1,      b_3-1,S\\
{[}( n\ m_2 \ m_1 \ m_3 ) \somecycle   b_3+1,  b_1+b_2-1,S\\
{[}( n\ m_3 \ m_1 \ m_2 ) \somecycle   b_1+1,  b_2+b_3-1,S\\
{[}( n\ m_1 \ m_3 \ m_2 ) \somecycle   b_1+b_2+b_3-1, 1,S\\
{[}( n\ m_2 \ m_4 \ m_3 ) \somecycle   b_1,     b_2+b_3,S\\
{[}( n\ m_2 \ m_3 \ m_4 ) \somecycle    2, b_1, b_2-1, b_3-1,S\\
{[}( n\ m_4 \ m_2 \ m_3 ) \somecycle    1, b_1, b_2-1, b_3,S\\
{[}( n\ m_4 \ m_3 \ m_2 ) \somecycle    1, 1,  b_1,   b_2+b_3-2,S\\
{[}( n\ m_3 \ m_2 \ m_4 ) \somecycle    1, b_1, b_2,   b_3-1,S
\end{eqnarray*}
(Note that the cycle lengths in these sequences on the right
hand sides are not in general
in increasing numerical order.)
In general, these will be distinct.  In special cases some of these
cycle types will be the same; but we can show that
there are at least $4$ different cycle types in all cases, and
if there are only $4$ cycle types, then we must have
either $b_2=b_3=2$ or $b_1=1,b_2=2,b_3=3$.

In these cases $b_2=2$ and $b_3\le 3$, so except possibly for one cycle of
length $3$, all cycles must have length at most $2$ (since $b_1,b_2,b_3$
were maximal length cycles in $\beta$), and since
$n\ge8$, there are at least $4$ cycles, so we can take
$$[\beta]=[R,b_4,b_1,b_2,b_3],$$ 
with $b_2,b_2,b_3$ as above, and $R$ a sequence, and
$b_4$ an integer, with $R,b_4=S$.  Then
$${[}(n\ m_5\ m_1\ m_2)\cdot\beta_2
(m_5\  \cdots)(m_1\ \cdots)(m_2\ m_3\ \cdots)(m_4\ \cdots n){]}=
b_1+b_2+b_3+b_4,R,$$
which is distinct from any of the above listed cycle types, since it
has few terms in this sequence.
Thus in this case $\eta'(\alpha,\beta)\ge5$.

Now assume that 
$\beta$ consists of 
either $2$ nontrivial cycles, and no fixed points
(since two nontrivial cycles and one fixed point is just the case 
of three cycles with $b_1=1$), or exactly two cycles, one of which may have
length one.
We write
$$\beta=\stand(\beta)=\beta_2(m_5\  \cdots)(m_2\ m_3\ m_4\ \cdots\ n),$$
with $m_2=n-b_2$ and $m_5=n-b_1-b_2$, and with
$\beta$ having cycle structure $S,b_1,b_2$.
Since $n\ge 8$, at least one cycle
has length at least $4$, so, including one possible case
where there is a cycle of length at least $5$, we have cycle types
\begin{eqnarray*}
{[}(n\ m_2\ m_3\ m_4)\cdot\beta_2(m_5\  \cdots)(n\ m_2\ m_3\ m_4\ \cdots){]}&=&2,b_1-2,b_2,S\\
{[}(n\ m_4\ m_3\ m_2)\cdot\beta_2(m_5\  \cdots)(n\ m_2\ m_3\ m_4\ \cdots){]}&=&1,1,1,b_1-3,b_2,S\\
{[}(n\ m_2\ m_4\ m_3)\cdot\beta_2(m_5\  \cdots)(n\ m_2\ m_3\ m_4\ \cdots){]}&=&1,b_1-1,b_2,S\\
{[}(n\ m_4\ m_2\ m_3)\cdot\beta_2(m_5\  \cdots)(n\ m_2\ m_3\ m_4\ \cdots){]}&=&3,b_1-3,b_2,S\\
{[}(n\ m_2\ m_3\ m_5)\cdot\beta_2(m_5\  \cdots)(n\ m_2\ m_3\ m_4\ \cdots){]}&=&b_1+b_2,S\\
{[}(n\ m_5\ m_3\ m_2)\cdot\beta_2(m_5\  \cdots)(n\ m_2\ m_3\ m_4\ \cdots){]}&=&1,1,b_1+b_2-2,b_3,S\\
{[}(1\ m_5\ m_4\ m_2)\cdot\beta_2(m_5\  \cdots)(n\ m_2\ m_3\ m_4\ m_5\ \cdots){]}&=&1,1,2,b_1-4,b_2,S
\end{eqnarray*}
Note that in the case $b_1=4$, the third and fourth line of the above
table have the same cycle type, but altogether there are $5$ different
cycle types.
If $\beta$ consists of only one cycle, 
then $\beta$ is a cycle of length $n$, then take the first four line
and the last line of the above table to give five different possible
cycle types of $\alpha^\sigma\beta$.

These cases cover all the possibilities, and so the result follows.
\end{proof}

\begin{prop}
\label{prop:reducetothreecases}
Suppose that $\alpha,\beta\in S_n$ for $n\ge 9$.
Suppose $\eta(\alpha^{S_n}\beta^{S_n})\le 5$.  Then one of $\alpha,\beta$ is
conjugate to $(), (1\ 2)$, or $(1\ 2\ 3)$.
\end{prop}
\begin{proof}
We may assume neither of $\alpha,\beta=e$.

We can check the result for $n=9$ with GAP.
See Table~\ref{table:eta_for_A9}
 in the Appendix. %, Subsection \ref{Sn tables}.

Suppose that $n>9$, then
by Proposition~\ref{prop:inductionstep}, we have, using the notation of
that lemma, 
$$\eta'(\alpha,\beta)\ge \eta'(\alpha',\beta').$$
If neither of $\alpha',\beta'$ is conjugate to
$(), (1\ 2), (1\ 2\ 3)$, then the result follows by induction.
The cases where $\alpha'$ has one of these forms are:
\begin{eqnarray*}
\alpha&=&(1\ 2)\\
\alpha&=&(1\ 2\ 3)\\
\alpha&=&(1\ 2\ 3\ 4),
\end{eqnarray*}
and similarly for $\beta$.
The first two cases are contained in the statement of the result.

Suppose that $\alpha=(1\ 2\ 3\ 4)$, and $\beta$ is not conjugate to 
$(), (1\ 2)$ or $(1\ 2\ 3)$.
Then by Lemma~\ref{lem:fourcyclemodifiedeta} 
to show that $\eta(\alpha^{S_n}\beta^{S_n})\ge 5$.
\end{proof}

\subsection{Cycle structures when $\eta\le 5$}

\begin{lemma}
\label{lem:beta_a_trans}
Suppose that $\beta$ is a transposition and 
$$[\alpha]=[
\underbrace{1,\cdots,1}_{n_0}, 
\underbrace{a_1,\cdots,a_1}_{n_1}, 
\underbrace{a_2,\cdots,a_2}_{n_2}, \ldots, 
\underbrace{a_r,\cdots,a_r}_{n_r}],$$
where $n_i\ge 1$ for $1\le i\le r$.  Let $r'=r$ if $n_0=0$ and
$r'=r+1$ otherwise.
Let $r''$ be the number of $n_i$, $0\le i\le r$, which are at least $2$.
Then
$$\eta(\alpha^{S_n}\beta^{S_n})=
\sum_{i=0}^r\left\lfloor\frac{a_i}{2}\right\rfloor
+\binom{r'}{2} + r''.$$
\end{lemma}
\begin{proof}
When computing $\eta(\alpha^{S_n}\beta^{S_n})$, we may just compute
the number of possible cycle types in
$\alpha^{S_n}\beta^\tau$ for some fixed $\tau$, and so we may assume
$\beta=(1\ 2)$.
Suppose that for some $\sigma$ we have
 $\alpha^\sigma=c_1c_2\cdots c_m$, where 
$c_i$ are disjoint cycles, and 
$m=\sum_{i=0}^r n_i$.  We may assume that
$c_1$ involves $1$, and either $c_1$ or $c_2$ involves $2$.

In the first case, $c_1$ may be a cycle of length $a_1, a_2,\dots a_r$.
If $c_1$ has length $a_i$, then $[(1\ 2)c_1]=[s, a_i-s]$, 
where $s$ depends on sigma, and has possible values
$1\le s< a_i$, so there are $\lfloor\frac{a_i}{2}\rfloor$
 possibilities for
the set $\{s, a_i-s\}$.
Now $\alpha^\sigma(1\ 2)$ has type
$[\alpha^\sigma(1\ 2)]=
T(i,s)$ where
$$
T(i,s):=
\{
\underbrace{1,\cdots,1}_{n_0}, 
\underbrace{a_1,\cdots,a_1}_{n_1},\ldots,
\underbrace{a_i,\cdots,a_i}_{n_i-1},s,a_i-s,
\ldots, 
\underbrace{a_r,\cdots,a_r}_{n_r}\}.$$
 As $a_i$ varies, taking values $a_1, a_2,\dots, a_r$,
and as $s$ varies in the range 
$1\le s\le \lfloor\frac{a_i}{2}\rfloor$,
we obtain $\sum_{i=1}^r\left\lfloor\frac{a_i}{2}\right\rfloor$
different possible types, which we claim are all distinct from each other.
To see this, suppose that we have
$T(i,s_1)=T(j,s_2)$.  Then we have the following equality of multisets:
$$\{s_1,a_i-s_1,a_j\}
=
\{a_i,s_2,a_j-s_2\}.$$
We now use the fact that 
$1\le s_1\le \lfloor\frac{a_i}{2}\rfloor$ and
$1\le s_2\le \lfloor\frac{a_j}{2}\rfloor$.
So we can't have $s_1=a_i$, and we can't have $a_i-s_1=a_i$.
Suppose $s_1=s_2$.  Then $a_i-s_1=a_j-s_2$, so $a_i=a_j$.
Suppose that $s_1=a_j-s_2$.  Then $a_i-s_1=s_2$ and $a_j=a_i$,
and $s_1+s_2=a_i$, which, since 
$1\le s_1,s_2\le \lfloor\frac{a_i}{2}\rfloor$, is only possible if
$s_1=s_2$.
So, in all cases, $a_i=a_j$ and $s_1=s_2$.

Now suppose that $c_1$ invoves $1$ and $c_2$ involves $2$.
There are two cases.  Either $c_1$ and $c_2$ both have length $a_i$
for some $i$, which is only possible if $n_i>1$, or else
$c_1$ and $c_2$ have lengths $a_i$, $a_j$ with $i\not=j$.
In the first case, there are $r''$ possible cycle types, and
in the second case, $\binom{r'}{2}$ cycle types.
These cycle types are all distinct from the types where $c_1$ involves
both $1$ and $2$, since then $\alpha^\sigma(1\ 2)$ consisted of 
$m+1$ distinct cycles, whereas in these cases there are $m-1$
distinct cycles.
The cycles in these cases are all distinct from each other, since if not,
we would have one of the following equalities of multisets:
\begin{eqnarray*}
\{
a_j,a_j,2a_i 
\}
&=&
\{
a_i,a_i,2a_j 
\}
\\
\{
a_j, 2a_i
\}
&=&
\{
a_i,a_i+a_j\}
\\
\{a_j,a_k,2a_i
\}
&=&
\{a_i,a_i,a_j+a_k
\}
\\
\{a_j,a_i+a_k\}
&=&
\{a_k,a_i+a_j\}
\\
\{a_k,a_l,a_i+a_j\}
&=&
\{a_i,a_j,a_k+a_l\},
\end{eqnarray*}
where $i,j,k,l$ are distinct.  But a consideration of each case shows that
none of these are possible, and so all cases are distinct, and the result
follows.
\end{proof}

\begin{cora}
If $\beta$ is a transposition, 
$\alpha\not=e$, and $\eta(\alpha^{S_n}\beta^{S_n})\le 5$,
then $\alpha$ has cycle type one of the following:
\begin{eqnarray*}
\{n\} &&\text{ for } 2\le n\le 11\\
\{\underbrace{i,\dots,i}_m\} &&\text{ for } 2\le i\le 9, m\ge 2
\text{ where }n=mi\\
\{i,j\}&&\text{ for }1\le i<j\le 8, i\le 4 \text{ and }n=i+j\le 9\\
\{3,7\}&&\text{ and } n=10\\
\{\underbrace{i,\dots,i}_{m},j\}
&&\text{ for } n=im + j\le 8, i\le 3, m\ge 2
\text{ and }\{i,j\}\not=\{2,6\}
\\
\{\underbrace{1,\dots,1}_{m_1},\underbrace{j,\dots,j}_{m_2}\}
&&\text{ for }m_1, m_2\ge 2, m_1 + jm_2=n\le 6, 2\le j\le 5\\
\{\underbrace{2,\dots,2}_{m_1},\underbrace{3,\dots,3}_{m_2}\}
&&\text{ for } 2m_1 + 3m_2=n\\
\{1,2,3\}&& n=6
\end{eqnarray*}
($1$s are not omitted from the cycle type.)
\end{cora}
\begin{proof}
By Lemma~\ref{lem:beta_a_trans}, we must have
$r'\le 3$.  

Suppose $r'=1$.  Then all cycles have the same length $a_1\ge 2$.
If $n_1=1$, then $r''=0$, and 
$\eta=\lfloor a_1/2\rfloor$, so $2\le a_1\le 11$.
This gives the first case in the above list.
If $n_1\ge2$, $r''=1$, and
$\eta=\lfloor a_1/2\rfloor+1$, so $2\le a_1\le 9$.  This gives the
second line.

Suppose $r'=2$, so there are two distinct cycle lengths, say
$a_i, a_j$, possibly $a_i=1$.
$\eta=\lfloor a_i/2\rfloor + 
\lfloor a_j/2\rfloor + 1 + r''$.
This gives the next $5$ cases.

If $r'=3$, the only possible case is the last one listed.
\end{proof}

\begin{lemma}
\label{lem:beta_a_3cycle}
Suppose that $\beta$ is a three cycle, and
$$[\alpha]=[
\underbrace{1,\cdots,1}_{n_0}, 
\underbrace{a_1,\cdots,a_1}_{n_1}, 
\underbrace{a_2,\cdots,a_2}_{n_2}, \ldots, 
\underbrace{a_r,\cdots,a_r}_{n_r}],$$
where $n_i\ge 1$ for $1\le i\le r$.  Let $r'$ and $r''$ be as in
Lemma~\ref{lem:beta_a_trans}.
Let $r'''$ be the number of $n_i$, $0\le i\le r$, which are at least $3$.
Then
\begin{eqnarray*}
\eta(\alpha^{S_n}\beta^{S_n})&\le& 
\sum_{i=0}^r P_3(a_i)
+\binom{r'}{3} + r''' + (r'-1)\sum\lfloor a_i/2\rfloor + s_3 \\
&&+\sum_{1\le i\le r,a_i\ge 2,n_i\ge 2} (a_i-1),
\end{eqnarray*}
where $P_3(a_i)$ is the number of partitions of $a_i$ into three parts,
and $s_3=1$ if any $a_i\ge 3$, and $0$ otherwise.
\end{lemma}
\begin{proof}
We may assume that $\beta=(1\ 2\ 3)$.
There are a number of possible cases, a few of which are:
$1,2,3$ are all involved in the same cycle $c$, of length $a_i$, 
of $\alpha^\sigma$,
in which case $[c\beta]$ can, amongst other possibilities,
be any partition of $a_i$ into $3$ parts. This gives the term $P_3(a_i)$.
Another possibility in this case is that $[c\beta]=\{a_i\}$.
This gives the term $s_3$.

Another possibility is that each of $1$, $2$, $3$ are in different cycles,
of different lengths.  This gives the term
$\binom{r'}{3}$ in the
sum.  A third case is when $1$, $2$, $3$ are all in different cycles, but
these cycles have the same length.  This gives the term $r'''$.

Another possible case is
$1, 2, 3$ are all in different cycles, and these have lengths
$a_i,a_i,a_j$ for $a_i\not=a_j$.  This gives the term
$(r'-1)\sum\lfloor a_i/2\rfloor$.

If $1$ and $2$ are both in the same cycle $c_1$, and $3$ is in a different
cycle, $c_2$, but $c_1$ and $c_2$ both have length $a_i$, then the type
of $c_1c_2(1\ 2\ 3)$ is $\{a_i+u,v\}$, where $u+v=a_i$, $u,v\ge 2$.
The term $\sum (a_i-1)$
comes from this case.

These cases can all shown to have distinct cycle type, in a similar
manner to the proof of Lemma~\ref{lem:beta_a_trans}.
There are a number of other cases, but for the inequality, we do not
need to consider them.
\end{proof}
\begin{remark}
\label{rem:partionsterms}
The first few terms in the increasing
sequence $P_3(a_i)$, for $i=1,2,3,\dots$ 
are $0, 0, 1, 1, 2, 3, 4, 5, 7.$
This is Sloane's sequence A069905 \cite{sloane}, where more terms and
a formula can be found.
\end{remark}
\begin{cora}
\label{cor:possible_eta_for_beta_3_cycle}
With $\beta$ a three cycle and $\alpha\in S_n\setminus\{e\}$, if
$n> 8$ and
$\eta(\alpha^{S_n}\beta^{S_n})\le 5$, then $\alpha$ has one of the
following cycle types:
$$
\{\underbrace{2,\dots,2}_{n_1}\},
\{\underbrace{3,\dots,3}_{n_1}\},
\{\underbrace{1,\dots,1}_{n_0}, 2\},
\{\underbrace{1,\dots,1}_{n_0}, 2,2\},
\{\underbrace{1,\dots,1}_{n_0}, 3\},
\{\underbrace{1,\dots,1}_{n_0}, 4\},
$$
where $n_0,n_1$ are integers, $n_1\ge 3$, $n_0\ge 1$.
In these cases, the values of 
$\eta(\alpha^{S_n}\beta^{S_n})$ are as follows respectively:
$2,4, 3,5,5,5$.
\end{cora}
\begin{proof}
We use the inequality in Lemma~\ref{lem:beta_a_3cycle}.
By Remark~\ref{rem:partionsterms}, 
we must have $a_i\le 8$.
We must also have $r'\le 4$, $r''\le 5$.
Let $R(\alpha)$ be the expression on the right of the inequality in
Lemma~\ref{lem:beta_a_3cycle}.
First suppose that all $n_i$ are $1$ or $0$.
A computer search for $\alpha$ with types
satisfying $R(\alpha)\le 5$ gives the following possible cases for
$[\alpha]$:
$$
\{1,3,3\},\{1,1,3,3\}
\{4\},\{4,4\},
\{5\},\{6\},\{7\},
\{1, 5\},\{1,1,5\},
\{2, 3\},\{2,2,3\}
\{2, 4\},
$$
$$
\{\underbrace{2,\dots,2}_{n_1}\},
\{\underbrace{3,\dots,3}_{n_1}\},
\{\underbrace{1,\dots,1}_{n_0}, \underbrace{2,\dots,2}_{n_1}\},
\{\underbrace{1,\dots,1}_{n_0}, 3\},
\{\underbrace{1,\dots,1}_{n_0}, 4\},
$$
where $n_0,n_1$ are arbitrary positive integers.
These have been listed so that only those in the second line satisfy $n>8$.

To obtain the final result, we go
through the cases with $n>8$ one at a time.
It is simple to see that representatives of all possible
conjucacy classes of the product of $\alpha^{S_n}\beta^{S_n}$ in these
cases are given as follows:

For $[\alpha]=\{\underbrace{2,\dots,2}_{n_1}\}$:
\begin{eqnarray*}
(1\ 2\ 3)(1\ 2)(3\ 4)(5\ 6)\alpha_1&=&(2\ 4\ 3)(5\ 6) \alpha_1\\
(1\ 3\ 5)(1\ 2)(3\ 4)(5\ 6)\alpha_1&=&(1\ 4\ 3\ 6\ 5\ 2)\alpha_1\\
\end{eqnarray*}
where $\alpha_1$ is a product of disjoint transpositions,
acting on $7,\dots,n$.

For $[\alpha]=
\{\underbrace{3,\dots,3}_{n_1}\}:
$
\begin{eqnarray*}
(1\ 2\ 3)(1\ 2\ 3)(4\ 5\ 6)(7\ 8\ 9)\alpha_1&=&(1\ 3\ 2)(4\ 5\ 6)(7\ 8\ 9)\alpha_1\\
(1\ 3\ 2)(1\ 2\ 3)(4\ 5\ 6)(7\ 8\ 9)\alpha_1&=&(1)(2)(3)(4\ 5\ 6)(7\ 8\ 9)\alpha_1\\
(1\ 2\ 4)(1\ 2\ 3)(4\ 5\ 6)(7\ 8\ 9)\alpha_1&=&(1\ 3)(2\ 5\ 6\ 4)(7\ 8\ 9)\alpha_1\\
(1\ 4\ 5)(1\ 2\ 3)(4\ 5\ 6)(7\ 8\ 9)\alpha_1&=&(4\ 6)(1\ 5\ 2\ 3)(7\ 8\ 9)\alpha_1\\
(1\ 3\ 4)(1\ 2\ 3)(4\ 5\ 6)(7\ 8\ 9)\alpha_1&=&(1)(2\ 3\ 5\ 6\ 4)(7\ 8\ 9)\alpha_1
\end{eqnarray*}
where $\alpha_1$ is a product of disjoint three cycles
acting on $10,\dots,n$

For $[\alpha]=
\{\underbrace{1,\dots,1}_{n_0},\underbrace{2,\dots,2}_{n_1}\}$:
\begin{eqnarray*}
(1\ 2\ 3)(4\ 5)\alpha_1&=&            (1\ 2\ 3)(4\ 5)\alpha_1 \\
(1\ 2\ 3)(3\ 4)\alpha_2&=&            (1\ 2\ 4\ 3)	  \alpha_2 \\
(1\ 2\ 3)(2\ 3)\alpha_3&=&            (1\ 3)	  \alpha_3 \\
(1\ 2\ 3)(2\ 4)(3\ 5)\alpha_1&=&      (1\ 4\ 2\ 5\ 3)  \alpha_1\\
(1\ 2\ 3)(1\ 2)(3\ 4)\alpha_2&=&      (2\ 4\ 3)	   \alpha_2\\
(1\ 2\ 3)(1\ 4)(2\ 5)(3\ 6)\alpha_0&=&(1\ 5\ 2\ 6\ 3\ 4)\alpha_0
\end{eqnarray*}
where $\alpha_0,\alpha_1,\alpha_2,\alpha_3$ have order $2$, 
and $\alpha_i$ fixes $1,2,\dots,6-i$.  The cases where
$\alpha$ permutes either $2$, $4$ or more than $6$ elements need
to be distinguished from each other to determine the exact value
of $\eta$ in each case.  For $\alpha$ a transposition, only
the first three possibilities occur.  For $\alpha$ a product of two
disjoint transpositions, the first $5$ cases occur.
If $\alpha$ is a product of three or more disjoint transpositions,
then all $6$ cases occur.

For $[\alpha]=\{\underbrace{1,\dots,1}_{n_0}, 3\}$:
\begin{eqnarray*}
(1\ 2\ 3)(4\ 5\ 6)&=&(1\ 2\ 3)(4\ 5\ 6)	\\
(1\ 2\ 3)(3\ 4\ 5)&=&(1\ 2\ 4\ 5\ 3)	\\
(1\ 2\ 3)(2\ 3\ 4)&=&(1\ 3)(2\ 4)	\\
(1\ 2\ 3)(3\ 2\ 4)&=&(1\ 4\ 3)	\\
(1\ 2\ 3)(1\ 2\ 3)&=&(1\ 3\ 2)	\\
(1\ 2\ 3)(1\ 3\ 2)&=&()
\end{eqnarray*}

For $[\alpha]=\{\underbrace{1,\dots,1}_{n_0}, 4\}$:
\begin{eqnarray*}
(1\ 2\ 3)(4\ 5\ 6\ 7)&=&(1\ 2\ 3)(4\ 5\ 6\ 7)  \\
(1\ 2\ 3)(3\ 4\ 5\ 6)&=&(1\ 2\ 4\ 5\ 6\ 3)	   \\
(1\ 2\ 3)(2\ 3\ 4\ 5)&=&(1\ 3)(2\ 4\ 5)	   \\
(1\ 2\ 3)(3\ 4\ 2\ 5)&=&(2\ 4)(1\ 5\ 3)	   \\
(1\ 2\ 3)(3\ 2\ 4\ 5)&=&(1\ 4\ 5\ 3)	   \\
(1\ 2\ 3)(1\ 2\ 3\ 4)&=&(1\ 3\ 2\ 4)	   \\
(1\ 2\ 3)(1\ 3\ 2\ 4)&=&(1\ 4)     
\end{eqnarray*}
\end{proof}

\comment{
pari commands, used in getting data for corollary:
ff(i,j,k)=(max(max(i,j),k)>2)
ff(i,j,k,l)=(max(max(max(i,j),k),l)>2)
gg(i)=if(i==0,0,i+1)
pp(a)=round(a^2/12)  \\ this is P_3(a)

\\ the following search assume there are at most 4 a_i's,
\\ and assumes the n_is are zero.
\\ the result shows there are at most 2 a_is.

\\rp is r prime
\\i,j,k,l are values of the a_i; if any are zero, omit them.
\\ni's can be 1,2, or >=3.  ``3'' counts as >=3

{
for(i=0,8,for(j=gg(i),8,for(k=gg(j),8,for(l=gg(k),8,
if(l==0,rp=0;as=[],
if(k==0,rp=1;as=[l],
if(j==0,rp=2;as=[k,l],if(i==0,rp=3;as=[j,k,l],rp=4;as=[i,j,k,l]))));
if(rp==0,print(0));
tot=binomial(rp,3) + ff(i,j,k,l);
for(ii=1,length(as),tot=tot+pp(as[ii]));
tot=tot+(rp-1)*sum(ii=1,length(as),floor(as[ii]/2));
if(tot<=5,print([i,j,k,l])
)))))
}

\\ here we assume there are 2 a_is, and work out the possible n_i
{
for(i=1,8,for(j=i+1,8,
tot=pp(i) + pp(j) + binomial(2,3) + ff(0,0,i,j) + floor(i/2) + floor(j/2);
for(n0=1,3,for(n1=1,3,
s1=0;s2=0;
if(n0==3,s1=1);if(n1==3,s2=1);rrr=s1+s2;
r1=0;r2=0;
if(n0>=2,r1=i-1);if(n1>=2,r2=j-1);
if(tot+rrr+r1+r2<=5,print([i,j,n0,n1]))
))))
}

\\ here we assume there is 1 a_i, and work out the possible n_i
{
for(i=1,8,
tot=pp(i) + binomial(1,3) + (i>2);
for(n0=1,3,
rrr=0;if(n0==3,rrr=1);
r1=0;if(n0>=2,r1=i-1);
if(tot+rrr+r1<=5,print([i,n0]))
))
}
}

\begin{proof}[Proof Of Theorem~\ref{thm:maintheorem}]
The minimum value of $\eta$ in $A_n$ is at most
$5$, since we always have $\eta((1\ 2\ 3)^{A_n}(1\ 2\ 3)^{A_n})=5$.

For $6 \leq n \leq 8$, the result follows from direct computation by GAP \cite{GAP4}, as 
can be seen in the tables in the Appendix. %Section~\ref{subsec:tables}.

Suppose $n \geq 9$.  From
Proposition~\ref{prop:reducefromAntoSn}, 
Proposition~\ref{prop:reducetothreecases}, and 
Corollary~\ref{cor:possible_eta_for_beta_3_cycle}, 
we see that the only possible cases with 
$\eta(\alpha^{A_n}\beta^{A_n})<5$ for
$\alpha,\beta\in A_n\setminus\{e\}$ are those given in the
statement of the theorem.
\end{proof}

\section{Appendix}
\label{appendix}

In this Appendix, we give tables of $\eta$ and $\eta'$ for $A_n$ and $S_n$ for 
$3\le n\le 9$, as well as a table of minimal values of $\eta$ up to $n=12$.
first we describe the GAP code used to produce these tables, which then follow.

%%%%%

\subsection{GAP code}

Given two elements $i$ and $j$ of a group $G$, the following code computes up to $n$ representatives of the classes appearing in the set $j^G i^G$.  Note that it suffices to only check products of the form $b i$, for $b \in j^G$, since any product $j^h i^g$ is conjugate to $j^{h g^{-1}} i$.
    
\begin{figure}[htpb]
\begin{verbatim}

conjproduct:=function(G,i,j,n)
  local b,c,cj,k,rl,rlen,ok;
  rl:=[];
  rlen:=0;
  cj:=ConjugacyClass(G,j);
  for b in cj do
    if rlen=n then break; fi;														#if found enough, stop
    ok:=true;
    k:=1;
    c:=i*b;
    while ok and k<=rlen do															#compare against found classes
      if IsConjugate(G,rl[k][2],c) then ok:=false; fi;	
      k:=k+1;
    od;
    if ok then																					#if new class, add to list
      Add(rl,[b,c]); 
      rlen:=rlen+1;
    fi;
  od;
  return rl;
end;

\end{verbatim}
\end{figure}

Suppose that $G$ has $k$ conjugacy classes, $cl_1, \ldots, cl_k$, and suppose that $n = m$, if $m \geq 1$, and otherwise, $n = k$. (For large $n$, it can take quite a bit of time to compute all values of $\eta$ for $S_n$ and $A_n$:  we can save time, if desired, by setting an upper bound $m$ for $\eta$.)  The following code then uses the previous code to generate 
\begin{itemize}
\item up to $n$ representatives of the classes appearing in the set $cl_j cl_i$, for each pair $j \leq i$, and
\item pairs $g,h$ of representatives, listed by $\min(\eta(h^G g^G),n)$.
\end{itemize}

\begin{figure}[htpb]
\begin{verbatim}

allconjproducts:=function(G,outputfile,m)

  local cl,ch,g,h,i,j,k,n,clist,clen,num,pair;

  cl:=ConjugacyClasses(G);
  k:=Length(cl);
  if m<1 then n:=k; else n:=m; fi;											#find at most n classes
  num:=[];
  for i in [1..n] do
    Add(num,[]);
  od;

  AppendTo(outputfile,"\r\n\r\n",G,"\r\n\r\n");
  for i in [1..k] do
    g := Representative(cl[i]);
    for j in [1..i] do
      h:=Representative(cl[j]);
      if h=() then																			#trivial h means eta = 1
        AppendTo(outputfile,"[",g,"]*[",h,"]: 1 \r\n",
                "[ ",[h,g]," ] \r\n\r\n");
        Add(num[1],[g,h]);
      else
        clist:=conjproduct(G,g,h,n);
        clen:=Length(clist);
        AppendTo(outputfile,"[",g,"]*[",h,"]: ",clen,"\r\n",
                clist,"\r\n\r\n");
        Add(num[clen],[g,h]);
      fi;
    od;
  od;

  for i in [1..n] do																		#construct list of pairs, ordered by eta
    AppendTo(outputfile,"\r\n eta=",i,":\r\n");
    for pair in num[i] do
      AppendTo(outputfile,pair,"\r\n");
    od;
  od;
  
  return;
  
end;

\end{verbatim}
\end{figure}

\eject

For example, 
\begin{verbatim} allconjproducts(AlternatingGroup(3),"A3.txt",0); \end{verbatim}
yields a text file with the following contents:

\begin{figure}[htpb]
\begin{verbatim}

AlternatingGroup( [ 1 .. 3 ] )

[()]*[()]: 1 
[ [ (), () ] ] 

[(1,2,3)]*[()]: 1 
[ [ (), (1,2,3) ] ] 

[(1,2,3)]*[(1,2,3)]: 1
[ [ (1,2,3), (1,3,2) ] ]

[(1,3,2)]*[()]: 1 
[ [ (), (1,3,2) ] ] 

[(1,3,2)]*[(1,2,3)]: 1
[ [ (1,2,3), () ] ]

[(1,3,2)]*[(1,3,2)]: 1
[ [ (1,3,2), (1,2,3) ] ]


 eta=1:
[ (), () ]
[ (1,2,3), () ]
[ (1,2,3), (1,2,3) ]
[ (1,3,2), () ]
[ (1,3,2), (1,2,3) ]
[ (1,3,2), (1,3,2) ]

 eta=2:

 eta=3:

\end{verbatim}
\end{figure}

%%%%%

%\subsection{Tables}

%\label{subsec:tables}
%\subsubsection{Minimum values of $\eta$ for $S_n$ for small $n$}
%\label{subsec:min_eta_for_Sn}
%This table shows all pairs where the minimum values of $\eta$ for $S_n$ are attained, for $3 \leq n \leq 12$.

\begin{table}[htpb]
$$
\begin{array}{|ccc|}
\hline
n & \min\eta &\text{pairs where attained}\\
\hline
3 & 1 & (1\ 2\ 3),  (1\ 2)\\
&& \\
4 & 1 & (1\ 2\ 3),  (1\ 2)(3\ 4)\\
&& \\
5 & 2 & (1\ 2\ 3\ 4\ 5),  (1\ 2)\\
&& \\
6 & 2 & (1\ 2)(3\ 4)(5\ 6),  (1\ 2)\\
&& (1\ 2\ 3),  (1\ 2)(3\ 4)(5\ 6)\\
&& (1\ 2\ 3)(4\ 5\ 6),  (1\ 2)\\
&& \\
7 & 3 & (1\ 2),  (1\ 2)\\
&& (1\ 2)(3\ 4)(5\ 6),  (1\ 2)\\
&& (1\ 2\ 3),  (1\ 2)\\
&& (1\ 2\ 3)(4\ 5\ 6),  (1\ 2)\\
&& (1\ 2\ 3\ 4\ 5\ 6\ 7),  (1\ 2)\\
&& \\
8 & 2 & (1\ 2)(3\ 4)(5\ 6)(7\ 8),  (1\ 2)\\
&& (1\ 2\ 3),  (1\ 2)(3\ 4)(5\ 6)(7\ 8)\\
&& \\
9 & 2 & (1\ 2\ 3)(4\ 5\ 6)(7\ 8\ 9),  (1\ 2)\\
&& \\
10 & 2 & ( 1\  2)( 3\  4)( 5\  6)( 7\  8)( 9\ 10),  ( 1\  2)\\
&& ( 1\  2\  3),  ( 1\  2)( 3\  4)( 5\  6)( 7\  8)( 9\ 10)\\
&& \\
11 & 3 & ( 1\  2),  ( 1\  2)\\
&& ( 1\  2)( 3\  4)( 5\  6)( 7\  8)( 9\ 10),  ( 1\  2)\\
&& ( 1\  2\  3),  ( 1\  2)\\
&& \\
12 & 2 & ( 1\  2)( 3\  4)( 5\  6)( 7\  8)( 9\ 10)(11\ 12),  ( 1\  2)\\
&& ( 1\  2\  3),  ( 1\  2)( 3\  4)( 5\  6)( 7\  8)( 9\ 10)(11\ 12)\\
&& ( 1\  2\  3)( 4\  5\  6)( 7\  8\  9)(10\ 11\ 12),  ( 1\  2)\\
\hline
\end{array}
$$
\caption{Minimum values of $\eta$ for $S_n$}
\label{table:min_eta_for_Sn}
\end{table}

%\subsubsection{Minimum values of $\eta$ for $A_n$ for small $n$}
%\label{subsec:min_eta_for_An}
%This table shows all pairs where the minimum values of $\eta$ for $A_n$ are attained, for $3 \leq n \leq 12$.
\begin{table}[htpb]
$$
\begin{array}{|ccc|}
\hline
n & \min\eta &\text{pairs where attained}\\
\hline
3 & 1 & (1\ 2\ 3),  (1\ 2\ 3)\\
&& (1\ 3\ 2),  (1\ 2\ 3)\\
&& (1\ 3\ 2),  (1\ 3\ 2)\\
&& \\
4 & 1 & (1\ 2\ 3),  (1\ 2)(3\ 4)\\
&& (1\ 2\ 3),  (1\ 2\ 3)\\
&& (1\ 2\ 4),  (1\ 2)(3\ 4)\\
&& (1\ 2\ 4),  (1\ 2\ 4)\\
&& \\
5 & 3 & (1\ 2\ 3\ 4\ 5),  (1\ 2)(3\ 4)\\
&& (1\ 2\ 3\ 5\ 4),  (1\ 2)(3\ 4)\\
&& \\
6 & 5 & (1\ 2\ 3),  (1\ 2)(3\ 4)\\
&& (1\ 2\ 3)(4\ 5\ 6),  (1\ 2)(3\ 4)\\
&& (1\ 2\ 3)(4\ 5\ 6),  (1\ 2\ 3)\\
&& (1\ 2\ 3\ 4)(5\ 6),  (1\ 2\ 3)\\
&& (1\ 2\ 3\ 4)(5\ 6),  (1\ 2\ 3)(4\ 5\ 6)\\
&& \\
7 & 5 & (1\ 2\ 3),  (1\ 2)(3\ 4)\\
&& (1\ 2\ 3),  (1\ 2\ 3)\\
&& \\
8 & 2 & (1\ 2\ 3),  (1\ 2)(3\ 4)(5\ 6)(7\ 8)\\
&& \\
9 & 4 & (1\ 2\ 3),  (1\ 2)(3\ 4)(5\ 6)(7\ 8)\\
&& \\
10 & 5 & ( 1\  2\  3),  ( 1\  2)( 3\  4)\\
&& ( 1\  2\  3),  ( 1\  2)( 3\  4)( 5\  6)( 7\  8)\\
&& ( 1\  2\  3),  ( 1\  2\  3)\\
&& \\
11 & 5 & ( 1\  2\  3),  ( 1\  2)( 3\  4)\\
&& ( 1\  2\  3),  ( 1\  2\  3)\\
&& \\
12 & 2 & ( 1\  2)( 3\  4)( 5\  6)( 7\  8)( 9\ 10)(11\ 12),  ( 1\  2\  3)\\
\hline
\end{array}
$$
\caption{Minimum values of $\eta$ for $A_n$}
\label{table:min_eta_for_An}
\end{table}

%\subsubsection{Values of $\eta$} Values of $\eta$ for $S_n$ and $A_n$ for $n = 3,4,5$.

\begin{table}[htpb]
%\begin{minipage}[b]{0.5\linewidth}
$$
\begin{array}{|@{\;}l@{\;}|@{\;}c@{\;}c@{\;}c@{\;}|} 
\hline
\raisebox{0.35cm}{\parbox{1.6cm}
{$\eta(\alpha^{S_3}\beta^{S_3})$}}
&\begin{sideways}()\end{sideways}&\begin{sideways}(1\ 2)\end{sideways}&\begin{sideways}(1\ 2\ 3)\;\end{sideways}\\ 
\hline
()&1&1&1\\ 
(1\ 2)&1&2&1\\ 
(1\ 2\ 3)&1&1&2\\
\hline
\end{array}
%$$
%\end{minipage}
\hspace{1in}
%\begin{minipage}[b]{0.5\linewidth}
%$$
\begin{array}{|@{\;}l@{\;}|@{\;}c@{\;}c@{\;}c@{\;}|} 
\hline
\raisebox{0.4cm}{\parbox{1.6cm}
{$\eta(\alpha^{A_3}\beta^{A_3})$}}
&\begin{sideways}()\end{sideways}&\begin{sideways}(1\ 2\ 3)\;\end{sideways}&\begin{sideways}(1\ 3\ 2)\end{sideways}\\ 
\hline
()&1&1&1\\ 
(1\ 2\ 3)&1&1&1\\ 
(1\ 3\ 2)&1&1&1\\
\hline
\end{array}
%\end{minipage}
$$
\caption{Values of $\eta$ for $S_3$ and $A_3$}
\label{table:eta_for_S3_and_A3}
\end{table}

%\eject

\begin{table}[htpb]
%\begin{minipage}[b]{0.5\linewidth}
$$
\begin{array}{|@{\;}l@{\;}|@{\;}c@{\;}c@{\;}c@{\;}c@{\;}c@{\;}|} 
\hline
\raisebox{0.6cm}{\parbox{1.6cm}
{$\eta(\alpha^{S_4}\beta^{S_4})$}}
&\begin{sideways}()\end{sideways}&\begin{sideways}(1\ 2)\end{sideways}&\begin{sideways}(1\ 2)(3\ 4)\;\end{sideways}&\begin{sideways}(1\ 2\ 3)\end{sideways}&\begin{sideways}(1\ 2\ 3\ 4)\end{sideways}\\ 
\hline
()&1&1&1&1&1\\ 
(1\ 2)&1&3&2&2&2\\ 
(1\ 2)(3\ 4)&1&2&2&1&2\\ 
(1\ 2\ 3)&1&2&1&3&2\\ 
(1\ 2\ 3\ 4)&1&2&2&2&3\\
\hline
\end{array}
%$$
%\end{minipage}\begin{minipage}[b]{0.5\linewidth}
%$$
\hspace{1in}
\begin{array}{|@{\;}l@{\;}|@{\;}c@{\;}c@{\;}c@{\;}c@{\;}|}
\hline
\raisebox{0.55cm}{\parbox{1.6cm}
{$\eta(\alpha^{A_4}\beta^{A_4})$}}
&\begin{sideways}()\end{sideways}&\begin{sideways}(1\ 2)(3\ 4)\;\end{sideways}&\begin{sideways}(1\ 2\ 3)\end{sideways}&\begin{sideways}(1\ 2\ 4)\end{sideways}\\ 
\hline
()&1&1&1&1\\ 
(1\ 2)(3\ 4)&1&2&1&1\\ 
(1\ 2\ 3)&1&1&1&2\\ 
(1\ 2\ 4)&1&1&2&1\\
\hline
\end{array}
%\end{minipage}
$$
\caption{Values of $\eta$ for $S_4$ and $A_4$}
\label{table:eta_for_S4_and_A4}
\end{table}

\begin{table}[htpb]
\begin{minipage}[b]{0.5\linewidth}
$$
\begin{array}{|@{\;}l@{\;}|@{\;}c@{\;}c@{\;}c@{\;}c@{\;}c@{\;}c@{\;}c@{\;}|}
\hline
\;\;\raisebox{0.7cm}{\parbox{1.6cm}{$\eta(\alpha^{S_5}\beta^{S_5})$}}
&\begin{sideways}()\end{sideways}&\begin{sideways}(1\ 2)\end{sideways}&\begin{sideways}(1\ 2)(3\ 4)\end{sideways}&\begin{sideways}(1\ 2\ 3)\end{sideways}&\begin{sideways}(1\ 2\ 3)(4\ 5)\;\end{sideways}&\begin{sideways}(1\ 2\ 3\ 4)\end{sideways}&\begin{sideways}(1\ 2\ 3\ 4\ 5)\end{sideways}\\ 
\hline
()&1&1&1&1&1&1&1\\ 
(1\ 2)&1&3&3&3&3&3&2\\ 
(1\ 2)(3\ 4)&1&3&4&3&3&3&3\\ 
(1\ 2\ 3)&1&3&3&4&3&3&3\\ 
(1\ 2\ 3)(4\ 5)&1&3&3&3&4&3&3\\ 
(1\ 2\ 3\ 4)&1&3&3&3&3&4&3\\ 
(1\ 2\ 3\ 4\ 5)&1&2&3&3&3&3&4\\
\hline
\end{array}
$$
\end{minipage}\begin{minipage}[b]{0.5\linewidth}
$$
\begin{array}{|@{\;}l@{\;}|@{\;}c@{\;}c@{\;}c@{\;}c@{\;}c@{\;}|} 
\hline
\:\raisebox{0.6cm}{\parbox{1.6cm}{$\eta(\alpha^{A_5}\beta^{A_5})$}}
&\begin{sideways}()\end{sideways}&\begin{sideways}(1\ 2)(3\ 4)\end{sideways}&\begin{sideways}(1\ 2\ 3)\end{sideways}&\begin{sideways}(1\ 2\ 3\ 4\ 5)\;\end{sideways}&\begin{sideways}(1\ 2\ 3\ 5\ 4)\end{sideways}\\ 
\hline
()&1&1&1&1&1\\ 
(1\ 2)(3\ 4)&1&5&4&3&3\\ 
(1\ 2\ 3)&1&4&5&4&4\\ 
(1\ 2\ 3\ 4\ 5)&1&3&4&4&4\\ 
(1\ 2\ 3\ 5\ 4)&1&3&4&4&4\\
\hline
\end{array}
$$
\end{minipage}
\caption{Values of $\eta$ for $S_5$ and $A_5$}
\label{table:eta_for_S5_and_A5}
\end{table}

%\subsubsection{Values of $\eta$} Values of $\eta$ for $S_n$ for $n = 6,7,8,9$.
%\label{Sn tables}

\begin{table}[htpb]
$$
\begin{array}{|@{\;}l@{\;}|@{\;}c@{\;}c@{\;}c@{\;}c@{\;}c@{\;}c@{\;}c@{\;}c@{\;}c@{\;}c@{\;}c@{\;}|} 
\hline
\;\;\;\raisebox{0.95cm}{\parbox{1.6cm}{$\eta(\alpha^{S_6}\beta^{S_6})$}}
&\begin{sideways}()\end{sideways}&\begin{sideways}(1\ 2)\end{sideways}&\begin{sideways}(1\ 2)(3\ 4)\end{sideways}&\begin{sideways}(1\ 2)(3\ 4)(5\ 6)\;\end{sideways}&\begin{sideways}(1\ 2\ 3)\end{sideways}&\begin{sideways}(1\ 2\ 3)(4\ 5)\end{sideways}&\begin{sideways}(1\ 2\ 3)(4\ 5\ 6)\end{sideways}&\begin{sideways}(1\ 2\ 3\ 4)\end{sideways}&\begin{sideways}(1\ 2\ 3\ 4)(5\ 6)\end{sideways}&\begin{sideways}(1\ 2\ 3\ 4\ 5)\end{sideways}&\begin{sideways}(1\ 2\ 3\ 4\ 5\ 6)\end{sideways}\\ 
\hline
()&1&1&1&1&1&1&1&1&1&1&1\\ 
(1\ 2)&1&3&4&2&3&5&2&4&4&3&3\\ 
(1\ 2)(3\ 4)&1&4&6&4&4&4&4&5&5&5&4\\ 
(1\ 2)(3\ 4)(5\ 6)&1&2&4&3&2&3&3&4&4&3&5\\ 
(1\ 2\ 3)&1&3&4&2&5&5&4&4&4&5&4\\ 
(1\ 2\ 3)(4\ 5)&1&5&4&3&5&6&4&5&5&5&5\\ 
(1\ 2\ 3)(4\ 5\ 6)&1&2&4&3&4&4&5&4&4&5&5\\ 
(1\ 2\ 3\ 4)&1&4&5&4&4&5&4&6&5&5&5\\ 
(1\ 2\ 3\ 4)(5\ 6)&1&4&5&4&4&5&4&5&6&5&5\\ 
(1\ 2\ 3\ 4\ 5)&1&3&5&3&5&5&5&5&5&6&5\\ 
(1\ 2\ 3\ 4\ 5\ 6)&1&3&4&5&4&5&5&5&5&5&6\\
\hline
\end{array}
$$
\caption{Values of $\eta$ for $S_6$}
\label{table:eta_for_S6}
\end{table}

\begin{table}[htpb]
$$
\begin{array}{|@{\;}l@{\;}|@{\;}c@{\;}c@{\;}c@{\;}c@{\;}c@{\;}c@{\;}c@{\;}c@{\;}c@{\;}c@{\;}c@{\;}c@{\;}c@{\;}c@{\;}c@{\;}|}
\hline
\;\;\;\;\;\raisebox{1.0cm}{\parbox{1.6cm}{$\eta(\alpha^{S_7}\beta^{S_7})$}}
&\begin{sideways}()\end{sideways}&\begin{sideways}(1\ 2)\end{sideways}&\begin{sideways}(1\ 2)(3\ 4)\end{sideways}&\begin{sideways}(1\ 2)(3\ 4)(5\ 6)\end{sideways}&\begin{sideways}(1\ 2\ 3)\end{sideways}&\begin{sideways}(1\ 2\ 3)(4\ 5)\end{sideways}&\begin{sideways}(1\ 2\ 3)(4\ 5)(6\ 7)\;\end{sideways}&\begin{sideways}(1\ 2\ 3)(4\ 5\ 6)\end{sideways}&\begin{sideways}(1\ 2\ 3\ 4)\end{sideways}&\begin{sideways}(1\ 2\ 3\ 4)(5\ 6)\end{sideways}&\begin{sideways}(1\ 2\ 3\ 4)(5\ 6\ 7)\end{sideways}&\begin{sideways}(1\ 2\ 3\ 4\ 5)\end{sideways}&\begin{sideways}(1\ 2\ 3\ 4\ 5)(6\ 7)\end{sideways}&\begin{sideways}(1\ 2\ 3\ 4\ 5\ 6)\end{sideways}&\begin{sideways}(1\ 2\ 3\ 4\ 5\ 6\ 7)\end{sideways}\\ 
\hline
()&1&1&1&1&1&1&1&1&1&1&1&1&1&1&1\\ 
(1\ 2)&1&3&4&3&3&6&4&3&4&6&4&4&4&4&3\\ 
(1\ 2)(3\ 4)&1&4&7&7&5&7&6&5&6&7&5&7&6&6&5\\ 
(1\ 2)(3\ 4)(5\ 6)&1&3&7&7&4&7&7&6&6&6&6&6&7&7&6\\ 
(1\ 2\ 3)&1&3&5&4&5&7&5&6&5&6&5&6&5&6&5\\ 
(1\ 2\ 3)(4\ 5)&1&6&7&7&7&8&7&7&7&7&7&7&7&7&6\\ 
(1\ 2\ 3)(4\ 5)(6\ 7)&1&4&6&7&5&7&8&6&5&7&7&6&7&6&7\\ 
(1\ 2\ 3)(4\ 5\ 6)&1&3&5&6&6&7&6&8&6&7&7&7&6&7&7\\ 
(1\ 2\ 3\ 4)&1&4&6&6&5&7&5&6&7&7&6&7&6&7&6\\ 
(1\ 2\ 3\ 4)(5\ 6)&1&6&7&6&6&7&7&7&7&8&7&7&7&7&7\\ 
(1\ 2\ 3\ 4)(5\ 6\ 7)&1&4&5&6&5&7&7&7&6&7&8&6&7&7&7\\ 
(1\ 2\ 3\ 4\ 5)&1&4&7&6&6&7&6&7&7&7&6&8&7&7&7\\ 
(1\ 2\ 3\ 4\ 5)(6\ 7)&1&4&6&7&5&7&7&6&6&7&7&7&8&7&7\\ 
(1\ 2\ 3\ 4\ 5\ 6)&1&4&6&7&6&7&6&7&7&7&7&7&7&8&7\\ 
(1\ 2\ 3\ 4\ 5\ 6\ 7)&1&3&5&6&5&6&7&7&6&7&7&7&7&7&8\\
\hline
\end{array}
$$
\caption{Values of $\eta$ for $S_7$}
\label{table:eta_for_S7}
\end{table}

\begin{table}[htpb]
\begin{small}
$$
\begin{array}{|@{\;}l@{\;}|@{\;}c@{\;}c@{\;}c@{\;}c@{\;}c@{\;}c@{\;}c@{\;}c@{\;}c@{\;}c@{\;}c@{\;}c@{\;}c@{\;}c@{\;}c@{\;}c@{\;}c@{\;}c@{\;}c@{\;}c@{\;}c@{\;}c@{\;}|}
\hline 
\;\;\;\;\;\;\;\raisebox{1.2cm}{\parbox{1.6cm}{$\eta(\alpha^{S_8}\beta^{S_8})$}}
&\begin{sideways}()\end{sideways}&\begin{sideways}(1\ 2)\end{sideways}&\begin{sideways}(1\ 2)(3\ 4)\end{sideways}&\begin{sideways}(1\ 2)(3\ 4)(5\ 6)\end{sideways}&\begin{sideways}(1\ 2)(3\ 4)(5\ 6)(7\ 8)\;\end{sideways}&\begin{sideways}(1\ 2\ 3)\end{sideways}&\begin{sideways}(1\ 2\ 3)(4\ 5)\end{sideways}&\begin{sideways}(1\ 2\ 3)(4\ 5)(6\ 7)\end{sideways}&\begin{sideways}(1\ 2\ 3)(4\ 5\ 6)\end{sideways}&\begin{sideways}(1\ 2\ 3)(4\ 5\ 6)(7\ 8)\end{sideways}&\begin{sideways}(1\ 2\ 3\ 4)\end{sideways}&\begin{sideways}(1\ 2\ 3\ 4)(5\ 6)\end{sideways}&\begin{sideways}(1\ 2\ 3\ 4)(5\ 6)(7\ 8)\end{sideways}&\begin{sideways}(1\ 2\ 3\ 4)(5\ 6\ 7)\end{sideways}&\begin{sideways}(1\ 2\ 3\ 4)(5\ 6\ 7\ 8)\end{sideways}&\begin{sideways}(1\ 2\ 3\ 4\ 5)\end{sideways}&\begin{sideways}(1\ 2\ 3\ 4\ 5)(6\ 7)\end{sideways}&\begin{sideways}(1\ 2\ 3\ 4\ 5)(6\ 7\ 8)\end{sideways}&\begin{sideways}(1\ 2\ 3\ 4\ 5\ 6)\end{sideways}&\begin{sideways}(1\ 2\ 3\ 4\ 5\ 6)(7\ 8)\end{sideways}&\begin{sideways}(1\ 2\ 3\ 4\ 5\ 6\ 7)\end{sideways}&\begin{sideways}(1\ 2\ 3\ 4\ 5\ 6\ 7\ 8)\end{sideways}\\ 
\hline
()&1&1&1&1&1&1&1&1&1&1&1&1&1&1&1&1&1&1&1&1&1&1\\ 
(1\ 2)&1&3&4&4&2&3&6&6&4&4&4&7&5&6&3&4&6&4&5&5&4&4\\ 
(1\ 2)(3\ 4)&1&4&8&9&5&5&9&10&9&8&7&11&9&8&8&8&9&7&9&9&8&6\\ 
(1\ 2)(3\ 4)(5\ 6)&1&4&9&12&7&5&10&10&9&10&8&10&11&9&8&9&10&8&11&9&9&9\\ 
(1\ 2)(3\ 4)(5\ 6)(7\ 8)&1&2&5&7&5&2&6&7&6&6&5&9&8&6&7&5&7&6&8&10&8&9\\ 
(1\ 2\ 3)&1&3&5&5&2&5&8&8&7&6&5&8&6&8&5&7&8&7&7&7&8&6\\ 
(1\ 2\ 3)(4\ 5)&1&6&9&10&6&8&12&10&10&10&9&10&10&11&8&10&11&9&11&9&9&9\\ 
(1\ 2\ 3)(4\ 5)(6\ 7)&1&6&10&10&7&8&10&12&10&10&8&11&10&10&9&10&10&11&9&11&11&9\\ 
(1\ 2\ 3)(4\ 5\ 6)&1&4&9&9&6&7&10&10&12&10&8&11&8&10&10&10&9&10&10&10&11&9\\ 
(1\ 2\ 3)(4\ 5\ 6)(7\ 8)&1&4&8&10&6&6&10&10&10&12&8&9&10&10&9&8&11&10&10&10&9&11\\ 
(1\ 2\ 3\ 4)&1&4&7&8&5&5&9&8&8&8&8&10&8&9&7&8&10&7&10&8&9&9\\ 
(1\ 2\ 3\ 4)(5\ 6)&1&7&11&10&9&8&10&11&11&9&10&12&10&10&11&11&10&10&10&11&11&9\\ 
(1\ 2\ 3\ 4)(5\ 6)(7\ 8)&1&5&9&11&8&6&10&10&8&10&8&10&12&11&10&8&10&9&10&10&9&11\\ 
(1\ 2\ 3\ 4)(5\ 6\ 7)&1&6&8&9&6&8&11&10&10&10&9&10&11&12&10&9&11&10&11&9&10&11\\ 
(1\ 2\ 3\ 4)(5\ 6\ 7\ 8)&1&3&8&8&7&5&8&9&10&9&7&11&10&10&12&8&9&10&9&11&11&10\\ 
(1\ 2\ 3\ 4\ 5)&1&4&8&9&5&7&10&10&10&8&8&11&8&9&8&11&10&10&10&10&11&9\\ 
(1\ 2\ 3\ 4\ 5)(6\ 7)&1&6&9&10&7&8&11&10&9&11&10&10&10&11&9&10&12&10&11&10&10&11\\ 
(1\ 2\ 3\ 4\ 5)(6\ 7\ 8)&1&4&7&8&6&7&9&11&10&10&7&10&9&10&10&10&10&12&9&11&11&10\\ 
(1\ 2\ 3\ 4\ 5\ 6)&1&5&9&11&8&7&11&9&10&10&10&10&10&11&9&10&11&9&12&10&10&11\\ 
(1\ 2\ 3\ 4\ 5\ 6)(7\ 8)&1&5&9&9&10&7&9&11&10&10&8&11&10&9&11&10&10&11&10&12&11&10\\ 
(1\ 2\ 3\ 4\ 5\ 6\ 7)&1&4&8&9&8&8&9&11&11&9&9&11&9&10&11&11&10&11&10&11&12&10\\ 
(1\ 2\ 3\ 4\ 5\ 6\ 7\ 8)&1&4&6&9&9&6&9&9&9&11&9&9&11&11&10&9&11&10&11&10&10&12\\
\hline
\end{array}
$$
\end{small}
\caption{Values of $\eta$ for $S_8$}
\label{table:eta_for_S8}
\end{table}

\begin{table}[htpb]
\begin{tiny}
$$
\begin{array}{|@{\;}l@{\;}|@{\;}c@{\;}c@{\;}c@{\;}c@{\;}c@{\;}c@{\;}c@{\;}c@{\;}c@{\;}c@{\;}c@{\;}c@{\;}c@{\;}c@{\;}c@{\;}c@{\;}c@{\;}c@{\;}c@{\;}c@{\;}c@{\;}c@{\;}c@{\;}c@{\;}c@{\;}c@{\;}c@{\;}c@{\;}c@{\;}c@{\;}|} 
\hline
\;\;\;\;\;\;\;\;\raisebox{1.0cm}{\parbox{1.6cm}{$\eta(\alpha^{S_9}\beta^{S_9})$}}
&\begin{sideways}()\end{sideways}&\begin{sideways}(1\ 2)\end{sideways}&\begin{sideways}(1\ 2)(3\ 4)\end{sideways}&\begin{sideways}(1\ 2)(3\ 4)(5\ 6)\end{sideways}&\begin{sideways}(1\ 2)(3\ 4)(5\ 6)(7\ 8)\end{sideways}&\begin{sideways}(1\ 2\ 3)\end{sideways}&\begin{sideways}(1\ 2\ 3)(4\ 5)\end{sideways}&\begin{sideways}(1\ 2\ 3)(4\ 5)(6\ 7)\end{sideways}&\begin{sideways}(1\ 2\ 3)(4\ 5)(6\ 7)(8\ 9)\;\end{sideways}&\begin{sideways}(1\ 2\ 3)(4\ 5\ 6)\end{sideways}&\begin{sideways}(1\ 2\ 3)(4\ 5\ 6)(7\ 8)\end{sideways}&\begin{sideways}(1\ 2\ 3)(4\ 5\ 6)(7\ 8\ 9)\end{sideways}&\begin{sideways}(1\ 2\ 3\ 4)\end{sideways}&\begin{sideways}(1\ 2\ 3\ 4)(5\ 6)\end{sideways}&\begin{sideways}(1\ 2\ 3\ 4)(5\ 6)(7\ 8)\end{sideways}&\begin{sideways}(1\ 2\ 3\ 4)(5\ 6\ 7)\end{sideways}&\begin{sideways}(1\ 2\ 3\ 4)(5\ 6\ 7)(8\ 9)\end{sideways}&\begin{sideways}(1\ 2\ 3\ 4)(5\ 6\ 7\ 8)\end{sideways}&\begin{sideways}(1\ 2\ 3\ 4\ 5)\end{sideways}&\begin{sideways}(1\ 2\ 3\ 4\ 5)(6\ 7)\end{sideways}&\begin{sideways}(1\ 2\ 3\ 4\ 5)(6\ 7)(8\ 9)\end{sideways}&\begin{sideways}(1\ 2\ 3\ 4\ 5)(6\ 7\ 8)\end{sideways}&\begin{sideways}(1\ 2\ 3\ 4\ 5)(6\ 7\ 8\ 9)\end{sideways}&\begin{sideways}(1\ 2\ 3\ 4\ 5\ 6)\end{sideways}&\begin{sideways}(1\ 2\ 3\ 4\ 5\ 6)(7\ 8)\end{sideways}&\begin{sideways}(1\ 2\ 3\ 4\ 5\ 6)(7\ 8\ 9)\end{sideways}&\begin{sideways}(1\ 2\ 3\ 4\ 5\ 6\ 7)\end{sideways}&\begin{sideways}(1\ 2\ 3\ 4\ 5\ 6\ 7)(8\ 9)\end{sideways}&\begin{sideways}(1\ 2\ 3\ 4\ 5\ 6\ 7\ 8)\end{sideways}&\begin{sideways}(1\ 2\ 3\ 4\ 5\ 6\ 7\ 8\ 9)\end{sideways}\\ 
\hline
()&1&1&1&1&1&1&1&1&1&1&1&1&1&1&1&1&1&1&1&1&1&1&1&1&1&1&1&1&1&1\\ 
(1\ 2)&1&3&4&4&3&3&6&7&4&4&6&2&4&7&7&7&7&4&4&7&5&6&5&5&7&5&5&5&5&4\\ 
(1\ 2)(3\ 4)&1&4&8&10&8&5&10&14&9&10&11&6&7&13&13&12&12&11&9&13&10&10&9&10&13&9&11&10&9&8\\ 
(1\ 2)(3\ 4)(5\ 6)&1&4&10&15&14&6&13&14&13&12&14&10&9&14&15&14&12&12&11&15&13&11&12&14&13&12&13&13&13&10\\ 
(1\ 2)(3\ 4)(5\ 6)(7\ 8)&1&3&8&14&12&4&12&14&14&11&13&7&7&13&14&12&13&12&9&13&14&12&12&12&14&12&13&13&13&13\\ 
(1\ 2\ 3)&1&3&5&6&4&5&8&10&6&8&9&5&5&9&9&10&9&7&7&10&7&11&7&8&10&8&9&8&9&8\\ 
(1\ 2\ 3)(4\ 5)&1&6&10&13&12&8&15&14&12&14&15&10&10&14&14&15&13&12&12&15&12&13&12&15&13&13&13&13&13&10\\ 
(1\ 2\ 3)(4\ 5)(6\ 7)&1&7&14&14&14&10&14&16&14&15&14&13&11&15&14&14&15&14&14&14&15&15&12&13&15&13&15&13&13&13\\ 
(1\ 2\ 3)(4\ 5)(6\ 7)(8\ 9)&1&4&9&13&14&6&12&14&16&11&14&12&8&12&15&14&14&12&10&14&14&12&14&12&13&15&12&15&13&13\\ 
(1\ 2\ 3)(4\ 5\ 6)&1&4&10&12&11&8&14&15&11&16&14&13&9&15&12&14&14&14&13&13&12&15&12&14&14&13&15&12&13&13\\ 
(1\ 2\ 3)(4\ 5\ 6)(7\ 8)&1&6&11&14&13&9&15&14&14&14&16&14&12&13&15&15&14&13&12&15&13&14&14&14&14&15&13&15&15&13\\ 
(1\ 2\ 3)(4\ 5\ 6)(7\ 8\ 9)&1&2&6&10&7&5&10&13&12&13&14&13&7&11&12&13&14&12&10&12&12&15&13&11&14&14&13&13&13&15\\ 
(1\ 2\ 3\ 4)&1&4&7&9&7&5&10&11&8&9&12&7&8&12&11&12&10&10&9&13&9&11&10&11&12&10&12&11&13&10\\ 
(1\ 2\ 3\ 4)(5\ 6)&1&7&13&14&13&9&14&15&12&15&13&11&12&16&14&14&15&15&14&14&14&14&13&14&15&12&15&13&13&13\\ 
(1\ 2\ 3\ 4)(5\ 6)(7\ 8)&1&7&13&15&14&9&14&14&15&12&15&12&11&14&16&15&14&14&12&15&14&13&15&14&14&14&13&15&15&13\\ 
(1\ 2\ 3\ 4)(5\ 6\ 7)&1&7&12&14&12&10&15&14&14&14&15&13&12&14&15&16&14&14&13&15&12&14&15&15&13&15&14&14&15&13\\ 
(1\ 2\ 3\ 4)(5\ 6\ 7)(8\ 9)&1&7&12&12&13&9&13&15&14&14&14&14&10&15&14&14&16&15&12&13&15&15&14&12&15&14&14&14&13&15\\ 
(1\ 2\ 3\ 4)(5\ 6\ 7\ 8)&1&4&11&12&12&7&12&14&12&14&13&12&10&15&14&14&15&16&12&13&14&15&14&13&15&13&15&13&14&15\\ 
(1\ 2\ 3\ 4\ 5)&1&4&9&11&9&7&12&14&10&13&12&10&9&14&12&13&12&12&12&14&12&14&11&13&14&11&15&12&13&13\\ 
(1\ 2\ 3\ 4\ 5)(6\ 7)&1&7&13&15&13&10&15&14&14&13&15&12&13&14&15&15&13&13&14&16&14&14&15&15&14&14&14&15&15&13\\ 
(1\ 2\ 3\ 4\ 5)(6\ 7)(8\ 9)&1&5&10&13&14&7&12&15&14&12&13&12&9&14&14&12&15&14&12&14&16&14&14&12&15&13&14&14&13&15\\ 
(1\ 2\ 3\ 4\ 5)(6\ 7\ 8)&1&6&10&11&12&11&13&15&12&15&14&15&11&14&13&14&15&15&14&14&14&16&14&13&15&14&15&13&14&15\\ 
(1\ 2\ 3\ 4\ 5)(6\ 7\ 8\ 9)&1&5&9&12&12&7&12&12&14&12&14&13&10&13&15&15&14&14&11&15&14&14&16&13&13&15&13&15&15&14\\ 
(1\ 2\ 3\ 4\ 5\ 6)&1&5&10&14&12&8&15&13&12&14&14&11&11&14&14&15&12&13&13&15&12&13&13&16&14&14&14&14&15&13\\ 
(1\ 2\ 3\ 4\ 5\ 6)(7\ 8)&1&7&13&13&14&10&13&15&13&14&14&14&12&15&14&13&15&15&14&14&15&15&13&14&16&14&15&14&14&15\\ 
(1\ 2\ 3\ 4\ 5\ 6)(7\ 8\ 9)&1&5&9&12&12&8&13&13&15&13&15&14&10&12&14&15&14&13&11&14&13&14&15&14&14&16&13&15&15&14\\ 
(1\ 2\ 3\ 4\ 5\ 6\ 7)&1&5&11&13&13&9&13&15&12&15&13&13&12&15&13&14&14&15&15&14&14&15&13&14&15&13&16&14&14&15\\ 
(1\ 2\ 3\ 4\ 5\ 6\ 7)(8\ 9)&1&5&10&13&13&8&13&13&15&12&15&13&11&13&15&14&14&13&12&15&14&13&15&14&14&15&14&16&15&14\\ 
(1\ 2\ 3\ 4\ 5\ 6\ 7\ 8)&1&5&9&13&13&9&13&13&13&13&15&13&13&13&15&15&13&14&13&15&13&14&15&15&14&15&14&15&16&14\\ 
(1\ 2\ 3\ 4\ 5\ 6\ 7\ 8\ 9)&1&4&8&10&13&8&10&13&13&13&13&15&10&13&13&13&15&15&13&13&15&15&14&13&15&14&15&14&14&16\\
\hline
\end{array}
$$
\end{tiny}
\caption{Values of $\eta$ for $S_9$}
\label{table:eta_for_S9}
\end{table}

%\subsubsection{Values of $\eta$} Values of $\eta$ for $A_n$ for $n = 6,7,8,9$.
%\label{An tables}

\begin{table}[htpb]
$$
\begin{array}{|@{\;}l@{\;}|@{\;}c@{\;}c@{\;}c@{\;}c@{\;}c@{\;}c@{\;}c@{\;}|} 
\hline
\;\;\:\raisebox{0.9cm}{\parbox{1.6cm}{$\eta(\alpha^{A_6}\beta^{A_6})$}}
&\begin{sideways}()\end{sideways}&\begin{sideways}(1\ 2)(3\ 4)\end{sideways}&\begin{sideways}(1\ 2\ 3)\end{sideways}&\begin{sideways}(1\ 2\ 3)(4\ 5\ 6)\end{sideways}&\begin{sideways}(1\ 2\ 3\ 4)(5\ 6)\;\end{sideways}&\begin{sideways}(1\ 2\ 3\ 4\ 5)\end{sideways}&\begin{sideways}(1\ 2\ 3\ 4\ 6)\end{sideways}\\ 
\hline
()&1&1&1&1&1&1&1\\ 
(1\ 2)(3\ 4)&1&7&5&5&6&6&6\\ 
(1\ 2\ 3)&1&5&6&5&5&6&6\\ 
(1\ 2\ 3)(4\ 5\ 6)&1&5&5&6&5&6&6\\ 
(1\ 2\ 3\ 4)(5\ 6)&1&6&5&5&7&6&6\\ 
(1\ 2\ 3\ 4\ 5)&1&6&6&6&6&7&6\\ 
(1\ 2\ 3\ 4\ 6)&1&6&6&6&6&6&7\\
\hline
\end{array}
$$
\caption{Values of $\eta$ for $A_6$}
\label{table:eta_for_A6}
\end{table}

\begin{table}[htpb]
$$
\begin{array}{|@{\;}l@{\;}|@{\;}c@{\;}c@{\;}c@{\;}c@{\;}c@{\;}c@{\;}c@{\;}c@{\;}c@{\;}|} 
\hline
\;\;\;\;\;\raisebox{1.1cm}{\parbox{1.6cm}{$\eta(\alpha^{A_7}\beta^{A_7})$}}
&\begin{sideways}()\end{sideways}&\begin{sideways}(1\ 2)(3\ 4)\end{sideways}&\begin{sideways}(1\ 2\ 3)\end{sideways}&\begin{sideways}(1\ 2\ 3)(4\ 5)(6\ 7)\;\end{sideways}&\begin{sideways}(1\ 2\ 3)(4\ 5\ 6)\end{sideways}&\begin{sideways}(1\ 2\ 3\ 4)(5\ 6)\end{sideways}&\begin{sideways}(1\ 2\ 3\ 4\ 5)\end{sideways}&\begin{sideways}(1\ 2\ 3\ 4\ 5\ 6\ 7)\end{sideways}&\begin{sideways}(1\ 2\ 3\ 4\ 5\ 7\ 6)\end{sideways}\\ 
\hline
()&1&1&1&1&1&1&1&1&1\\ 
(1\ 2)(3\ 4)&1&7&5&7&6&8&8&6&6\\ 
(1\ 2\ 3)&1&5&5&6&7&7&7&6&6\\ 
(1\ 2\ 3)(4\ 5)(6\ 7)&1&7&6&9&7&8&7&8&8\\ 
(1\ 2\ 3)(4\ 5\ 6)&1&6&7&7&9&8&8&8&8\\ 
(1\ 2\ 3\ 4)(5\ 6)&1&8&7&8&8&9&8&8&8\\ 
(1\ 2\ 3\ 4\ 5)&1&8&7&7&8&8&9&8&8\\ 
(1\ 2\ 3\ 4\ 5\ 6\ 7)&1&6&6&8&8&8&8&8&9\\ 
(1\ 2\ 3\ 4\ 5\ 7\ 6)&1&6&6&8&8&8&8&9&8\\
\hline
\end{array}
$$
\caption{Values of $\eta$ for $A_7$}
\label{table:eta_for_A7}
\end{table}

\begin{table}[htpb]
$$
\begin{array}{|@{\;}l@{\;}|@{\;}c@{\;}c@{\;}c@{\;}c@{\;}c@{\;}c@{\;}c@{\;}c@{\;}c@{\;}c@{\;}c@{\;}c@{\;}c@{\;}c@{\;}|} 
\hline
\;\;\;\;\;\;\;\raisebox{1.3cm}{\parbox{1.6cm}{$\eta(\alpha^{A_8}\beta^{A_8})$}}
&\begin{sideways}()\end{sideways}&\begin{sideways}(1\ 2)(3\ 4)\end{sideways}&\begin{sideways}(1\ 2)(3\ 4)(5\ 6)(7\ 8)\;\end{sideways}&\begin{sideways}(1\ 2\ 3)\end{sideways}&\begin{sideways}(1\ 2\ 3)(4\ 5)(6\ 7)\end{sideways}&\begin{sideways}(1\ 2\ 3)(4\ 5\ 6)\end{sideways}&\begin{sideways}(1\ 2\ 3\ 4)(5\ 6)\end{sideways}&\begin{sideways}(1\ 2\ 3\ 4)(5\ 6\ 7\ 8)\end{sideways}&\begin{sideways}(1\ 2\ 3\ 4\ 5)\end{sideways}&\begin{sideways}(1\ 2\ 3\ 4\ 5)(6\ 7\ 8)\end{sideways}&\begin{sideways}(1\ 2\ 3\ 4\ 5)(6\ 8\ 7)\end{sideways}&\begin{sideways}(1\ 2\ 3\ 4\ 5\ 6)(7\ 8)\end{sideways}&\begin{sideways}(1\ 2\ 3\ 4\ 5\ 6\ 7)\end{sideways}&\begin{sideways}(1\ 2\ 3\ 4\ 5\ 6\ 8)\end{sideways}\\ 
\hline
()&1&1&1&1&1&1&1&1&1&1&1&1&1&1\\ 
(1\ 2)(3\ 4)&1&8&5&5&12&11&13&10&9&9&9&11&10&10\\ 
(1\ 2)(3\ 4)(5\ 6)(7\ 8)&1&5&5&2&9&7&11&8&7&7&7&12&10&10\\ 
(1\ 2\ 3)&1&5&2&5&10&9&9&7&9&9&9&9&10&10\\ 
(1\ 2\ 3)(4\ 5)(6\ 7)&1&12&9&10&14&12&13&11&12&13&13&13&13&13\\ 
(1\ 2\ 3)(4\ 5\ 6)&1&11&7&9&12&14&13&12&12&12&12&12&13&13\\ 
(1\ 2\ 3\ 4)(5\ 6)&1&13&11&9&13&13&14&13&13&12&12&13&13&13\\ 
(1\ 2\ 3\ 4)(5\ 6\ 7\ 8)&1&10&8&7&11&12&13&14&10&12&12&13&13&13\\ 
(1\ 2\ 3\ 4\ 5)&1&9&7&9&12&12&13&10&13&12&12&12&13&13\\ 
(1\ 2\ 3\ 4\ 5)(6\ 7\ 8)&1&9&7&9&13&12&12&12&12&13&13&13&13&13\\ 
(1\ 2\ 3\ 4\ 5)(6\ 8\ 7)&1&9&7&9&13&12&12&12&12&13&13&13&13&13\\ 
(1\ 2\ 3\ 4\ 5\ 6)(7\ 8)&1&11&12&9&13&12&13&13&12&13&13&14&13&13\\ 
(1\ 2\ 3\ 4\ 5\ 6\ 7)&1&10&10&10&13&13&13&13&13&13&13&13&13&14\\ 
(1\ 2\ 3\ 4\ 5\ 6\ 8)&1&10&10&10&13&13&13&13&13&13&13&13&14&13\\
\hline
\end{array}
$$
\caption{Values of $\eta$ for $A_8$}
\label{table:eta_for_A8}
\end{table}

\begin{table}[htpb]
$$
\begin{array}{|@{\;}l@{\;}|@{\;}c@{\;}c@{\;}c@{\;}c@{\;}c@{\;}c@{\;}c@{\;}c@{\;}c@{\;}c@{\;}c@{\;}c@{\;}c@{\;}c@{\;}c@{\;}c@{\;}c@{\;}c@{\;}|} 
\hline
\;\;\;\;\;\;\;\;\raisebox{1.4cm}{\parbox{1.6cm}{$\eta(\alpha^{A_9}\beta^{A_9})$}}
&\begin{sideways}()\end{sideways}&\begin{sideways}(1\ 2)(3\ 4)\end{sideways}&\begin{sideways}(1\ 2)(3\ 4)(5\ 6)(7\ 8)\end{sideways}&\begin{sideways}(1\ 2\ 3)\end{sideways}&\begin{sideways}(1\ 2\ 3)(4\ 5)(6\ 7)\end{sideways}&\begin{sideways}(1\ 2\ 3)(4\ 5\ 6)\end{sideways}&\begin{sideways}(1\ 2\ 3)(4\ 5\ 6)(7\ 8\ 9)\;\end{sideways}&\begin{sideways}(1\ 2\ 3\ 4)(5\ 6)\end{sideways}&\begin{sideways}(1\ 2\ 3\ 4)(5\ 6\ 7)(8\ 9)\end{sideways}&\begin{sideways}(1\ 2\ 3\ 4)(5\ 6\ 7\ 8)\end{sideways}&\begin{sideways}(1\ 2\ 3\ 4\ 5)\end{sideways}&\begin{sideways}(1\ 2\ 3\ 4\ 5)(6\ 7)(8\ 9)\end{sideways}&\begin{sideways}(1\ 2\ 3\ 4\ 5)(6\ 7\ 8)\end{sideways}&\begin{sideways}(1\ 2\ 3\ 4\ 5)(6\ 7\ 9)\end{sideways}&\begin{sideways}(1\ 2\ 3\ 4\ 5\ 6)(7\ 8)\end{sideways}&\begin{sideways}(1\ 2\ 3\ 4\ 5\ 6\ 7)\end{sideways}&\begin{sideways}(1\ 2\ 3\ 4\ 5\ 6\ 7\ 8\ 9)\end{sideways}&\begin{sideways}(1\ 2\ 3\ 4\ 5\ 6\ 7\ 9\ 8)\end{sideways}\\ 
\hline
()&1&1&1&1&1&1&1&1&1&1&1&1&1&1&1&1&1&1\\ 
(1\ 2)(3\ 4)&1&8&8&5&15&11&8&14&14&13&9&11&12&12&15&13&10&10\\ 
(1\ 2)(3\ 4)(5\ 6)(7\ 8)&1&8&13&4&16&13&9&15&15&14&11&16&13&13&16&15&15&15\\ 
(1\ 2\ 3)&1&5&4&5&11&9&7&9&11&9&8&9&13&13&12&11&10&10\\ 
(1\ 2\ 3)(4\ 5)(6\ 7)&1&15&16&11&18&17&15&17&17&16&16&17&17&17&17&17&15&15\\ 
(1\ 2\ 3)(4\ 5\ 6)&1&11&13&9&17&18&15&17&16&16&15&14&17&17&16&17&15&15\\ 
(1\ 2\ 3)(4\ 5\ 6)(7\ 8\ 9)&1&8&9&7&15&15&15&13&16&14&12&14&17&17&16&15&17&17\\ 
(1\ 2\ 3\ 4)(5\ 6)&1&14&15&9&17&17&13&18&17&17&16&16&16&16&17&17&15&15\\ 
(1\ 2\ 3\ 4)(5\ 6\ 7)(8\ 9)&1&14&15&11&17&16&16&17&18&17&14&17&17&17&17&16&17&17\\ 
(1\ 2\ 3\ 4)(5\ 6\ 7\ 8)&1&13&14&9&16&16&14&17&17&18&14&16&17&17&17&17&17&17\\ 
(1\ 2\ 3\ 4\ 5)&1&9&11&8&16&15&12&16&14&14&14&14&16&16&16&17&15&15\\ 
(1\ 2\ 3\ 4\ 5)(6\ 7)(8\ 9)&1&11&16&9&17&14&14&16&17&16&14&18&16&16&17&16&17&17\\ 
(1\ 2\ 3\ 4\ 5)(6\ 7\ 8)&1&12&13&13&17&17&17&16&17&17&16&16&17&17&17&17&17&17\\ 
(1\ 2\ 3\ 4\ 5)(6\ 7\ 9)&1&12&13&13&17&17&17&16&17&17&16&16&17&17&17&17&17&17\\ 
(1\ 2\ 3\ 4\ 5\ 6)(7\ 8)&1&15&16&12&17&16&16&17&17&17&16&17&17&17&18&17&17&17\\ 
(1\ 2\ 3\ 4\ 5\ 6\ 7)&1&13&15&11&17&17&15&17&16&17&17&16&17&17&17&18&17&17\\ 
(1\ 2\ 3\ 4\ 5\ 6\ 7\ 8\ 9)&1&10&15&10&15&15&17&15&17&17&15&17&17&17&17&17&18&17\\ 
(1\ 2\ 3\ 4\ 5\ 6\ 7\ 9\ 8)&1&10&15&10&15&15&17&15&17&17&15&17&17&17&17&17&17&18\\
\hline
\end{array}
$$
\caption{Values of $\eta$ for $A_9$}
\label{table:eta_for_A9}
\end{table}

%\subsubsection{Values of $\eta'$} Values of $\eta'$ for $S_n$ for $n = 6,7,8,9$.
%\label{eta' tables}

\begin{table}[htpb]
$$
\begin{array}{|@{\;}l@{\;}|@{\;}c@{\;}c@{\;}c@{\;}c@{\;}c@{\;}c@{\;}|}
\hline 
\;\;\;\;\;\;\;\;\raisebox{0.9cm}{\parbox{1.6cm}{$\eta'(\alpha, \beta)$}}
&\begin{sideways}()\end{sideways}&\begin{sideways}(1\ 2)(3\ 4)\end{sideways}&\begin{sideways}(1\ 2\ 3)\end{sideways}&\begin{sideways}(1\ 2\ 3)(4\ 5\ 6)\end{sideways}&\begin{sideways}(1\ 2\ 3\ 4)(5\ 6)\end{sideways}&\begin{sideways}(1\ 2\ 3\ 4\ 5)\end{sideways}\\ 
\hline
()&1&1&1&1&1&1\\ 
(1\ 2)(3\ 4)&1&6&4&4&5&5\\ 
(1\ 2\ 3)&1&4&4&4&4&5\\ 
(1\ 2\ 3)(4\ 5\ 6)&1&4&4&5&4&5\\ 
(1\ 2\ 3\ 4)(5\ 6)&1&5&4&4&6&5\\ 
(1\ 2\ 3\ 4\ 5)&1&5&5&5&5&6\\
\hline
\end{array}
$$
\caption{Values of $\eta'$ for even permutations of $S_6$}
\label{table:eta'_for_evenS6}
\end{table}

\begin{table}[htpb]
$$
\begin{array}{|@{\;}l@{\;}|@{\;}c@{\;}c@{\;}c@{\;}c@{\;}c@{\;}c@{\;}c@{\;}c@{\;}|} 
\hline
\;\;\;\;\;\;\;\;\raisebox{1.0cm}{\parbox{1.6cm}{$\eta'(\alpha, \beta)$}}
&\begin{sideways}()\end{sideways}&\begin{sideways}(1\ 2)(3\ 4)\end{sideways}&\begin{sideways}(1\ 2\ 3)\end{sideways}&\begin{sideways}(1\ 2\ 3)(4\ 5)(6\ 7)\end{sideways}&\begin{sideways}(1\ 2\ 3)(4\ 5\ 6)\end{sideways}&\begin{sideways}(1\ 2\ 3\ 4)(5\ 6)\end{sideways}&\begin{sideways}(1\ 2\ 3\ 4\ 5)\end{sideways}&\begin{sideways}(1\ 2\ 3\ 4\ 5\ 6\ 7)\end{sideways}\\ 
\hline
()&1&1&1&1&1&1&1&1\\ 
(1\ 2)(3\ 4)&1&7&4&5&5&7&7&5\\ 
(1\ 2\ 3)&1&4&4&4&6&6&6&5\\ 
(1\ 2\ 3)(4\ 5)(6\ 7)&1&5&4&8&6&7&6&7\\ 
(1\ 2\ 3)(4\ 5\ 6)&1&5&6&6&8&7&7&7\\ 
(1\ 2\ 3\ 4)(5\ 6)&1&7&6&7&7&8&7&7\\ 
(1\ 2\ 3\ 4\ 5)&1&7&6&6&7&7&8&7\\ 
(1\ 2\ 3\ 4\ 5\ 6\ 7)&1&5&5&7&7&7&7&8\\
\hline
\end{array}
$$
\caption{Values of $\eta'$ for even permutations of $S_7$}
\label{table:eta'_for_evenS7}
\end{table}

\begin{table}[htpb]
$$
\begin{array}{|@{\;}l@{\;}|@{\;}c@{\;}c@{\;}c@{\;}c@{\;}c@{\;}c@{\;}c@{\;}c@{\;}c@{\;}c@{\;}c@{\;}c@{\;}|} 
\hline
\;\;\;\;\;\;\;\;\raisebox{1.1cm}{\parbox{1.6cm}{$\eta'(\alpha, \beta)$}}
&\begin{sideways}()\end{sideways}&\begin{sideways}(1\ 2)(3\ 4)\end{sideways}&\begin{sideways}(1\ 2)(3\ 4)(5\ 6)(7\ 8)\end{sideways}&\begin{sideways}(1\ 2\ 3)\end{sideways}&\begin{sideways}(1\ 2\ 3)(4\ 5)(6\ 7)\end{sideways}&\begin{sideways}(1\ 2\ 3)(4\ 5\ 6)\end{sideways}&\begin{sideways}(1\ 2\ 3\ 4)(5\ 6)\end{sideways}&\begin{sideways}(1\ 2\ 3\ 4)(5\ 6\ 7\ 8)\end{sideways}&\begin{sideways}(1\ 2\ 3\ 4\ 5)\end{sideways}&\begin{sideways}(1\ 2\ 3\ 4\ 5)(6\ 7\ 8)\end{sideways}&\begin{sideways}(1\ 2\ 3\ 4\ 5\ 6)(7\ 8)\end{sideways}&\begin{sideways}(1\ 2\ 3\ 4\ 5\ 6\ 7)\end{sideways}\\ 
\hline
()&1&1&1&1&1&1&1&1&1&1&1&1\\ 
(1\ 2)(3\ 4)&1&7&5&4&8&9&10&8&8&7&9&8\\ 
(1\ 2)(3\ 4)(5\ 6)(7\ 8)&1&5&5&2&7&6&9&7&5&6&10&8\\ 
(1\ 2\ 3)&1&4&2&4&6&7&7&5&6&6&7&8\\ 
(1\ 2\ 3)(4\ 5)(6\ 7)&1&8&7&6&12&10&11&9&10&10&11&11\\ 
(1\ 2\ 3)(4\ 5\ 6)&1&9&6&7&10&12&11&10&10&10&10&11\\ 
(1\ 2\ 3\ 4)(5\ 6)&1&10&9&7&11&11&12&11&11&10&11&11\\ 
(1\ 2\ 3\ 4)(5\ 6\ 7\ 8)&1&8&7&5&9&10&11&12&8&10&11&11\\ 
(1\ 2\ 3\ 4\ 5)&1&8&5&6&10&10&11&8&11&10&10&11\\ 
(1\ 2\ 3\ 4\ 5)(6\ 7\ 8)&1&7&6&6&10&10&10&10&10&12&11&11\\ 
(1\ 2\ 3\ 4\ 5\ 6)(7\ 8)&1&9&10&7&11&10&11&11&10&11&12&11\\ 
(1\ 2\ 3\ 4\ 5\ 6\ 7)&1&8&8&8&11&11&11&11&11&11&11&12\\
\hline
\end{array}
$$
\caption{Values of $\eta'$ for even permutations of $S_8$}
\label{table:eta'_for_evenS8}
\end{table}

\begin{table}[htpb]
$$
\begin{array}{|@{\;}l@{\;}|@{\;}c@{\;}c@{\;}c@{\;}c@{\;}c@{\;}c@{\;}c@{\;}c@{\;}c@{\;}c@{\;}c@{\;}c@{\;}c@{\;}c@{\;}c@{\;}c@{\;}|}
\hline 
\;\;\;\;\;\;\;\;\raisebox{1.2cm}{\parbox{1.6cm}{$\eta'(\alpha, \beta)$}}
&\begin{sideways}()\end{sideways}&\begin{sideways}(1\ 2)(3\ 4)\end{sideways}&\begin{sideways}(1\ 2)(3\ 4)(5\ 6)(7\ 8)\end{sideways}&\begin{sideways}(1\ 2\ 3)\end{sideways}&\begin{sideways}(1\ 2\ 3)(4\ 5)(6\ 7)\end{sideways}&\begin{sideways}(1\ 2\ 3)(4\ 5\ 6)\end{sideways}&\begin{sideways}(1\ 2\ 3)(4\ 5\ 6)(7\ 8\ 9)\end{sideways}&\begin{sideways}(1\ 2\ 3\ 4)(5\ 6)\end{sideways}&\begin{sideways}(1\ 2\ 3\ 4)(5\ 6\ 7)(8\ 9)\end{sideways}&\begin{sideways}(1\ 2\ 3\ 4)(5\ 6\ 7\ 8)\end{sideways}&\begin{sideways}(1\ 2\ 3\ 4\ 5)\end{sideways}&\begin{sideways}(1\ 2\ 3\ 4\ 5)(6\ 7)(8\ 9)\end{sideways}&\begin{sideways}(1\ 2\ 3\ 4\ 5)(6\ 7\ 8)\end{sideways}&\begin{sideways}(1\ 2\ 3\ 4\ 5\ 6)(7\ 8)\end{sideways}&\begin{sideways}(1\ 2\ 3\ 4\ 5\ 6\ 7)\end{sideways}&\begin{sideways}(1\ 2\ 3\ 4\ 5\ 6\ 7\ 8\ 9)\end{sideways}\\ 
\hline
()&1&1&1&1&1&1&1&1&1&1&1&1&1&1&1&1\\ 
(1\ 2)(3\ 4)&1&7&8&4&11&10&6&11&10&11&8&9&10&13&11&8\\ 
(1\ 2)(3\ 4)(5\ 6)(7\ 8)&1&8&12&4&14&11&7&13&13&12&9&14&12&14&13&13\\ 
(1\ 2\ 3)&1&4&4&4&7&7&5&7&7&7&6&6&9&10&9&8\\ 
(1\ 2\ 3)(4\ 5)(6\ 7)&1&11&14&7&16&14&13&15&14&14&13&15&14&15&15&13\\ 
(1\ 2\ 3)(4\ 5\ 6)&1&10&11&7&14&16&13&15&14&14&13&12&15&14&15&13\\ 
(1\ 2\ 3)(4\ 5\ 6)(7\ 8\ 9)&1&6&7&5&13&13&13&11&14&12&10&12&15&14&13&15\\ 
(1\ 2\ 3\ 4)(5\ 6)&1&11&13&7&15&15&11&16&15&15&14&14&14&15&15&13\\ 
(1\ 2\ 3\ 4)(5\ 6\ 7)(8\ 9)&1&10&13&7&14&14&14&15&16&15&12&15&15&15&14&15\\ 
(1\ 2\ 3\ 4)(5\ 6\ 7\ 8)&1&11&12&7&14&14&12&15&15&16&12&14&15&15&15&15\\ 
(1\ 2\ 3\ 4\ 5)&1&8&9&6&13&13&10&14&12&12&12&12&14&14&15&13\\ 
(1\ 2\ 3\ 4\ 5)(6\ 7)(8\ 9)&1&9&14&6&15&12&12&14&15&14&12&16&14&15&14&15\\ 
(1\ 2\ 3\ 4\ 5)(6\ 7\ 8)&1&10&12&9&14&15&15&14&15&15&14&14&16&15&15&15\\ 
(1\ 2\ 3\ 4\ 5\ 6)(7\ 8)&1&13&14&10&15&14&14&15&15&15&14&15&15&16&15&15\\ 
(1\ 2\ 3\ 4\ 5\ 6\ 7)&1&11&13&9&15&15&13&15&14&15&15&14&15&15&16&15\\ 
(1\ 2\ 3\ 4\ 5\ 6\ 7\ 8\ 9)&1&8&13&8&13&13&15&13&15&15&13&15&15&15&15&16\\
\hline
\end{array}
$$
\caption{Values of $\eta'$ for even permutations of $S_9$}
\label{table:eta'_for_evenS9}
\end{table}

\eject

\end{document}